\documentclass[12pt]{amsart}
\usepackage{amssymb,amscd}
\let\frak\mathfrak
\let\Bbb\mathbb

\def\>{\relax\ifmmode\mskip.666667\thinmuskip\relax\else\kern.111111em\fi}
\def\<{\relax\ifmmode\mskip-.333333\thinmuskip\relax\else\kern-.0555556em\fi}
\def\vsk#1>{\vskip#1\baselineskip}
\def\vv#1>{\vadjust{\vsk#1>}\ignorespaces}
\def\vvn#1>{\vadjust{\nobreak\vsk#1>\nobreak}\ignorespaces}

\let\Medskip\medskip
\def\medskip{\par\Medskip}
\let\Bigskip\bigskip
\def\bigskip{\par\Bigskip}

\let\Maketitle\maketitle
\def\maketitle{\Maketitle\thispagestyle{empty}\let\maketitle\empty}

\newtheorem{thm}{Theorem}[section]
\newtheorem{cor}[thm]{Corollary}
\newtheorem{lem}[thm]{Lemma}
\newtheorem{prop}[thm]{Proposition}

\numberwithin{equation}{section}

\theoremstyle{definition}
\newtheorem*{rem}{Remark}

\let\mc\mathcal
\let\nc\newcommand

\nc{\on}{\operatorname}
\nc{\Z}{{\mathbb Z}}
\nc{\C}{{\mathbb C}}
\nc{\N}{{\mathbb N}}
\nc{\pone}{{\mathbb C}{\mathbb P}^1}
\nc{\arr}{\rightarrow}
\nc{\larr}{\longrightarrow}
\nc{\al}{\alpha}
\nc{\W}{{\mc W}}
\nc{\la}{\lambda}
\nc{\su}{\widehat{{\mathfrak sl}}_2}
\nc{\g}{{\mathfrak g}}
\nc{\h}{{\mathfrak h}}
\nc{\m}{{\mathfrak m}}
\nc{\n}{{\mathfrak n}}
\nc{\Gm}{\Gamma}
\nc{\La}{\Lambda}
\nc{\gl}{\widehat{\mathfrak{gl}_2}}
\nc{\bi}{\bibitem}
\nc{\om}{\omega}
\nc{\Res}{\on{Res}}
\nc{\gm}{\gamma}
\nc{\Om}{\Omega}
\nc{\yy}{{\bs y}}
\nc{\kk}{{\bs k}}
\nc{\Tau}{\mathrm{T}}

\def\z{\mathfrak z}
\def\zA{\z(A_{2n}^{(1)})}
\def\zAA{\z(A_{2n}^{(2)})}

\def\Res{\on{Res}}

\def\Wr{\on{Wr}}

\def\B{{\mc B}}
\def\D{{\mc D}}

\def\F{{\mc F}}
\def\L{{\mc L}}
\def\M{{\mc M}}

\def\V{{\mc V}}

\let\Dl\Delta

\let\si\sigma
\let\Si\Sigma

\let\der\partial
\let\minus\setminus

\let\ge\geqslant
\let\geq\geqslant

\let\leq\leqslant

\nc{\gln}{\mathfrak{gl}_N}
\nc{\sln}{\mathfrak{sl}_N}

\def\beq{\begin{equation}}
\def\eeq{\end{equation}}
\def\be{\begin{equation*}}
\def\ee{\end{equation*}}

\nc{\bean}{\begin{eqnarray}}
\nc{\eean}{\end{eqnarray}}
\nc{\bea}{\begin{eqnarray*}}
	\nc{\eea}{\end{eqnarray*}}
\nc{\bs}{\boldsymbol}
\nc{\Ref}[1]{{\rm(\ref{#1})}}

\nc{\R}{\Bbb R}
\nc{\glN}{\mathfrak{gl}_N}
\nc{\glNt}{\mathfrak{gl}_N[t]}
\nc{\s}{sing}

\nc{\Oml}{{\Om_{\bs\la}}}
\nc{\OmLb}{{\Om_{\bs\La,\bs\la,\bs b}}}
\nc{\Ol}{{\mc O_{\bs\la}}}
\nc{\OLb}{{\mc O_{\bs\La,\bs\la,\bs b}}}
\nc{\Ml}{{\mc M_{\bs\la}}}
\nc{\Mlb}{{\mc M_{\bs\La,\bs\la,\bs b}}}
\nc{\Blb}{{\B_{\bs\La,\bs\la,\bs b}}}
\nc{\Omn}{{\Omega_{\bs n,\bs b,\bs K}}}
\nc{\Omlb}{{\bar\Om_{\bs\la}}}
\nc{\VSl}{{(\V^S)_{\bs\la}}}

\nc{\Dlb}{\Dl_{\bs\La,\bs\la,\bs b,\bs K}}
\nc{\ep}{\epsilon}
\nc{\Vn}{{V^{\otimes n}}}
\nc{\Il}{{\mc I_{\bs\la}}}
\nc{\bla}{{\bs\la}}
\nc{\Fla}{\F_{\bs\la}}
\nc{\GL}{{GL_n(\C)}}
\nc{\ga}{\gamma}
\nc{\Ga}{\Gamma}

\def\Bb{{\mc B}}
\nc{\Nn}{{\mc N}}
\nc\Ll{{\mc L}}
\def\Wr{\on{Wr}}
\nc{\PCN}{{   (\C[x])^2   }}
\nc{\slt}{{\frak{sl}_{2n+1}}}
\nc\ad{{\on{ad}}}
\nc\gA{{\g(A_{2n}^{(1)})}}
\nc{\gAA}{{\g(A_{2n}^{(2)})}}
\nc\At{{A_{2n}^{(1)}}}
\nc\Dia{{\on{Diag}}}
\nc\8{\emptyset}
\nc\AT{{A^{(2)}_{2n}}}

\nc\ox{{\otimes}}

\begin{document}
	
	\hrule width0pt
	\vsk->

	\title[ Critical points and mKdV hierarchy of type   $A^{(2)}_{2n}$]
{Critical points of master functions and mKdV hierarchy of type    $A^{(2)}_{2n}$}

\author[Alexander Varchenko,  Tyler Woodruff]
{Alexander Varchenko$^\star$, Tyler Woodruff$^\diamond$}

	\maketitle

	\begin{center}
		{\it Department of Mathematics, University of North Carolina
			at Chapel Hill\\ Chapel Hill, NC 27599-3250, USA\/}
	\end{center}

	{\let\thefootnote\relax
		\footnotetext{\vsk-.8>\noindent
			$^\star$ E-mail: anv@email.unc.edu, supported in part by NSF grant DMS-1362924 and Simons Foundation
\\  
grant   \#336826
\\			$^\diamond$ E-mail: tykwood@gmail.com}}
	
	\medskip
	
	\begin{abstract}
We consider the population of critical points
generated from the  critical point of the master function with no variables, which is
associated with the trivial representation of the twisted affine Lie algebra $\AT$.
The population is naturally partitioned into an infinite collection of complex cells $\C^m$, where $m$ are some positive integers.
For each  cell we define an injective rational map $\C^m \to \mc M(\AT)$ of the cell to the space $\mc M(\AT)$
of Miura opers of type $\AT$. We show that the image of the map is invariant with respect 
to all mKdV flows on $\mc M(\AT)$ and the image is point-wise fixed by all  mKdV
flows $\frac\der{\der t_r}$ with index $r$ greater than $4m$.

\end{abstract}
	
   {\small \tableofcontents	}

	\setcounter{footnote}{0}
	\renewcommand{\thefootnote}{\arabic{footnote}}

	\section{Introduction}

Let $\g$ be a Kac-Moody algebra with invariant scalar product $(\,,\,)$,
 $\h\subset \g$ Cartan subalgebra,
  $\al_0,\dots,\al_n$ simple roots. Let
$\Lambda_1,\dots,\Lambda_N$ be dominant integral weights,
 $k_0,\dots, k_n$
nonnegative integers, $k=k_0+\dots+k_n$.

Consider $\C^N$ with coordinates $z=(z_1,\dots,z_N)$.
Consider $\C^k$ with coordinates $u$ collected into $n+1$ groups, the $j$-th group consisting of $k_j$ variables,
\bea
u=(u^{(0)},\dots,u^{(n)}),
\qquad
u^{(j)} = (u^{(j)}_1,\dots,u^{(j)}_{k_j}).
\eea
The {\it master function} is the multivalued function on $\C^k\times\C^N$ defined by the formula
\bean
\label{mmff}
&&
\Phi(u,z) = \sum_{a<b} (\La_a,\La_b) \ln (z_a-z_b)
-  \sum_{a,i,j} (\al_j,\La_a)\ln (u^{(j)}_i-z_a) +
\\
\notag
&&
+ \sum_{j< j'} \sum_{i,i'} (\al_j,\al_{j'})
\ln (u^{(j)}_i-u^{(j')}_{i'})
+  \sum_{j} \sum_{i<i'} (\al_j,\al_{j})
\ln (u^{(j)}_i-u^{(j)}_{i'}),
\eean
with singularities at the places where the arguments of the logarithms are equal to zero.

A point in $\C^k\times \C^N$ can be interpreted as a collection of particles in $\C$:\ $z_a, u^{(j)}_i$.
A particle $z_a$ has weight $\La_a$, a particle $u^{(j)}_i$ has weight $-\al_j$.
The particles interact pairwise. The interaction of two particles is determined by
the scalar product of their weights.
The master function is the "total energy" of the collection of particles.

Notice that all scalar products are integers. So the master function is the logarithm
of a rational function. From a "physical" point of view, all interactions are integer multiples of
a certain unit of measurement. This is important for what will follow.

The variables $u$ are the {\it true} variables, variables $z$ are {\it parameters}.
We may think that the positions of $z$-particles are fixed and the $u$-particles can
move.

There are "global" characteristics of this situation,
\bea
I(z,\kappa) = \int e^{\Phi(u,z)/\kappa} A(u,z) du ,
\eea
where $A(u,z)$ is a suitable density function, $\kappa$ a parameter,
and
there are "local" characteristics  -- critical points
of the master function with respect to the $u$-variables,
\bea
d_u \Phi(u,z)=0 .
\eea
A critical point is an equilibrium position of the
$u$-particles  for fixed positions of the $z$-particles. In this paper we are interested in the equilibrium positions of
the $u$-particles.

Examples of master functions associated with  $\g=\frak{sl}_2$
were considered by Stieltjes and Heine in 19th century, see for example \cite{Sz}.
Master functions we introduced in \cite{SV}
to construct integral representations for solutions of the KZ equations, see also \cite{V1, V2}.

The critical points of master functions with respect to the $u$-variables were used  to
find eigenvectors in the associated Gaudin models by the Bethe ansatz method,
see   \cite{BF, RV, V3}. In important cases the algebra of functions on the critical set of a
master function is closely related to Schubert calculus, see \cite{MTV}.

In \cite{ScV, MV1} it was observed that the critical points of master functions with respect to the $u$-variables
can be deformed and
form families. Having one critical point, one can construct a family of new critical points. The family is
called a population of critical points. A point of the population is a critical point of the same master function
or of another master function associated with the same $\g,\La_1,\dots,\La_N$ but with  different integer parameters
$k_0,\dots,k_n$. The population is a variety isomorphic to the flag variety of the Kac-Moody algebra $\g^t$ Langlands dual
to $\g$, see \cite{MV1, MV2, F}.

In \cite{VW}, it was discovered that the population originated from the critical point of the master function associated
with the affine Lie algebra $\widehat{\frak{sl}}_{n+1}$ and the parameters $N=0, k_0=\dots=k_n=0$ is connected with the mKdV
integrable hierarchy associated with $\widehat{\frak{sl}}_{n+1}$. Namely, that population can be naturally embedded into the
space of $\widehat{\frak{sl}}_{n+1}$ Miura opers so that the image of the embedding is invariant with respect
to all mKdV flows on the space of Miura opers. For $n=1$, that result follows from  the classical paper by M.\,Adler and J.\,Moser \cite{AM},
which served as a motivation for \cite{VW}.

In this  paper we prove the analogous statement  for the twisted affine Lie algebra $\AT$. The special case  $A^{(2)}_2$ was considered in
\cite{VWW}.

\smallskip
In Sections \ref{sec KM1} - \ref{sec MKDV} we follow the paper \cite{DS} by V.\,Drinfled and V.\,Sokolov.  We
 review the affine Lie algebras $A^{(1)}_{2n}$ and $\AT$, the associated mKdV and KdV hierarchies,  Miura maps. 
For example, the $\AT$ mKdV hierachy is a sequence of commuting flows on the infinite-dimensional space
$ \mc M(\AT)$ of the $\AT$ Miura opers.

In Section \ref{tmaps} we study tangent maps to Miura maps.
In Section \ref{sec gene} formula  \Ref{Master}, we introduce our master functions,
\bean
\label{mmmfff}
&& \Phi(u^{(0)}_1,\dots,u^{(0)}_{k_0},\dots, u^{(n)}_1,\dots, u^{(n)}_{k_n})
\\
&&
\notag
\phantom{aa}
=
		-4 \sum_{i,i'}
		\ln (u^{(n-1)}_i-u^{(n)}_{i'})-
		2\sum_{j=0}^{n-2} \sum_{i,i'} 
		\ln (u^{(j)}_i-u^{(j+1)}_{i'})
\\
&&
\notag
\phantom{aaaa}
 +  8\sum_{i<i'} 
		\ln (u^{(n)}_i-u^{(n)}_{i'})+4\sum_{j=1}^{n-1} \sum_{i<i'} 
		\ln (u^{(j)}_i-u^{(j)}_{i'})
		+2\sum_{i<i'} \ln (u^{(0)}_i-u^{(0)}_{i'}).		
\eean
This master function is the special case of the master function in \Ref{mmff}. The master function in \Ref{mmmfff} is defined by formula
\Ref{mmff} if $\g$ is the Langlands dual to $\AT$ and  $N=0$, see a remark in Section \ref{mM}.

 Following \cite{MV1, MV2, VW}, we describe
the generation procedure of new critical points starting from a given one. We define the population of critical points
generated from the critical point of the function with no variables, namely, the function corresponding to the parameters 
$k_0=\dots = k_n=0$.
That population is partitioned into complex cells $\C^m$ labeled by degree increasing sequences $J=(j_1,\dots,i_m)$,
see the definition in Section \ref{sec degree transf}. 

In Theorem \ref{sec degree transf} we deduce from \cite{MV3}
that every critical point of the master function in \Ref{mmmfff} with arbitrary parameters $k_0,\dots,k_n$  belongs a cell of our population.
Moreover, a function in \Ref{mmmfff} with some parameters $k_0,\dots,k_n$ either does not have critical points at all or its critical points form
a cell $\C^m$ corresponding to a degree increasing sequence.

In Section \ref{sec cr and Miu}, to every degree increasing sequence $J$ we assign a rational injective
map $\mu^J : \C^m \to \mc M(\AT)$ of the cell corresponding to $J$ to the space
$\mc M(\AT)$ of Miura opers of type $\AT$. We describe  properties of that map.

In Section \ref{sec Vector fields}, we formulate and prove our main result. Theorem \ref{thm main} says that for any degree increasing  sequence,
the variety $\mu^J(\C^m)$ is invariant with respect to all mKdV flows on $\mc M(\AT)$ and that variety is point-wise fixed by all
flows $\frac\der{\der t_r}$ with index $r$ greater than $4m$.

Our result shows that there is a deep interrelation between the critical set of the master functions of the form \Ref{mmmfff}
and rational finite-dimensional 
submanifolds of the space
$\mc M(\AT)$, invariant with respect to all flows of the $\AT$ mKdV hierarchy.

 Initially the critical points of the master functions were related to
quantum integrable systems of the Gaudin type through the Bethe ansatz, \cite{SV, BF, RV, V3}.  Our result shows
 that the critical points are also related to the classical integrable systems, namely, the mKdV hierarchies.

In the next papers we plan to extend this result to other affine Lie algebras.

\smallskip
The first author thanks  MPI in Bonn for hospitality in 2015-2016.

	\section{Kac-Moody algebra of type $A_{2n}^{(1)}$}
	\label{sec KM1}
	In this section we follow \cite[Section 5]{DS}.
	
	\subsection{Definition}

	For $n\geq 2$, consider the $(2n+1) \times (2n+1) $ Cartan matrix of type  $A_{2n}^{(1)}$,
	\bea
	A_{2n}^{(1)}= \left(
	\begin{matrix}
		a_{0,0} & a_{0,1}  & \dots & a_{0,2n} \\
		\dots  & \dots  & \dots  & \dots \\
		\dots  & \dots  & \dots  & \dots \\
		a_{2n,0}  & a_{2n,1}  & \dots & a_{2n,2n}
	\end{matrix}
	\right)=\left(
	\begin{matrix}
		2 & -1  & 0 & \dots & 0 &  -1\\
		-1  & 2  & -1& \dots & 0 & 0 \\
		0 & -1  & 2 & \dots & \dots & \dots \\
		\dots & \dots & \dots & \dots & \dots & \dots \\
		0 & 0 & 0 &\dots &2  &-1  \\
		-1 & 0 &0 & \dots &-1 & 2 \\
	\end{matrix}
	\right).
	\eea
	The Kac-Moody algebra $\gA$ {\it of type} $A_{2n}^{(1)}$ is the Lie algebra with {\it canonical generators} $E_i,H_i,F_i\in\gA, i=0, \dots, 2n$,
	subject to the relations:
	\begin{eqnarray*}
	&
	[E_i,F_j]=\delta_{i,j}H_i,
	\\
	&
	[H_i, E_j] = a_{i,j} E_j , 
	\qquad
	[H_i,  F_j] = - a_{i,j} F_j,
	\qquad
	(\on{ad} E_i)^{1-a_{i,j}} E_j=0,
	\\
	&
	(\on{ad}  F_i)^{1-a_{i,j}} F_j=0,
	\qquad
	[H_i,H_j] = 0,
	\qquad
	\sum_{i=0}^{2n}H_{i}=0,
	\end{eqnarray*}
	see \cite[Section 5]{DS}.
	The Lie algebra $\gA$ is graded with respect to the {\it standard grading}, $\deg E_i=1, \deg F_i=-1$,\ $ i=0, \dots,2n $.
	Let $\gA^j=\{x\in\g(A_{2n}^{(1)})\ |\ \deg x=j\}$, then $\gA = \oplus_{j\in\Z}\,\gA^j$.
	
	Notice that $\g(A_{2n}^{(1)})^0$ is the $2n$-dimensional space generated by the $H_i$.
	Denote $\h=\g(A_{2n}^{(1)})^0$. Introduce elements $\al_j$ of the dual space $\h^*$ by the conditions
	$\langle \al_j, H_i\rangle =a_{i,j}$ for $i,j=0,\dots,2n$.
	For $j=0,1,\dots 2n$, we denote by $\n_j^-\subset \gA$ the Lie subalgebra generated by $F_i$, $i\in\{0,1,\dots, 2n\},\, i\ne j$.
	For example, $\n^-_0$ is generated by $F_1,F_2,\dots, F_{2n}$.

	\subsection{Realizations of $\gA$}
	\label{Realizations of gA}

	Consider the complex Lie algebra $\frak{sl}_{2n+1}$ with standard basis $e_{i,j}$,\ $i,j=1, \dots,2n+1$.
	
	Let $w=e^{2\pi i/{(2n+1)}}$.
	Define the {\it Coxeter automorphism} $C : \frak{sl}_{2n+1}\to \frak{sl}_{2n+1} $  of order $2n+1$ by the formula
	\bea
	C(X) = SXS^{-1},\ \ S = \on{diag}(1,w,\dots, w^{2n}).
	\eea
	Denote $(\frak{sl}_{2n+1})_j=\{ x \in \frak{sl}_{2n+1}\ | \ Cx=w^jx\}$.
	The twisted Lie subalgebra $L(\frak{sl}_{2n+1}, C)\subset\frak{sl}_{2n+1}[\xi,\xi^{-1}]$ is the subalgebra
	\bea
	L(\frak{sl}_{2n+1}, C) = \oplus_{j \in \mathbb{Z}}\, \xi^{j} \otimes (\frak{sl}_{2n+1})_{j\, \on{mod}\, 2n +1} .
	\eea
	The isomorphism $\tau_C: \gA \to L(\frak{sl}_{2n+1}, C)$ is defined by the formula, for $i = 1,\dots, 2n$
	\bea
	&&E_0\mapsto \xi\otimes e_{1,2n+1},
	\qquad
	E_i\mapsto \xi\otimes e_{i+1,i},
	\\
	&&
	F_0\mapsto \xi^{-1}\otimes e_{2n+1,1},
	\qquad
	F_i\mapsto \xi^{-1}\otimes e_{i,i+1},
	\\
	&&
	H_0\mapsto 1\otimes (e_{1,1}-e_{2n+1,2n+1}),
	\qquad
	H_i\mapsto 1\otimes (-e_{i,i}+e_{i+1,i+1}).
	\eea
	Under this isomorphism we have $\gA^j=\xi^j\otimes (\slt)_j$.
	
	\medskip
	
	The {\it standard automorphism} $\sigma_0 :\frak{sl}_{2n+1}\to \frak{sl}_{2n+1} $  is the identity: $\sigma_{0}(X) = X$.
	The isomorphism $\tau_0: \gA \to L(\frak{sl}_{2n+1}, \si_0)$ is defined by the formula, for $i = 1,\dots, 2n$
	\bea
	&&
	E_0\mapsto \la \otimes e_{1,2n+1},
	\qquad
	E_i\mapsto 1 \otimes e_{i+1,i},
	\\
	&&
	F_0\mapsto \lambda^{-1}\otimes e_{2n+1,1},
	\qquad
	F_i\mapsto1 \otimes e_{i,i+1},
	\\
	&&
	H_0\mapsto 1\otimes (e_{1,1}-e_{2n+1,2n+1}),
	\qquad
	H_i\mapsto 1\otimes (-e_{i,i}+e_{i+1,i+1}).
	\eea
	
	\subsection{Element $\La^{(1)}$}
	Denote by $\La^{(1)}$ the element $\sum_{j=0}^{2n}{E_{j}}\in \gA$. 
	Then $\zA=\{x\in\gA\ | \ [\La^{(1)},x]=0\}$ is an Abelian Lie subalgebra of $\g(A_{2n}^{(1)})$. Denote $\zA^j = \zA \cap \gA^j$, then
	$\zA =\oplus_{j\in\Z}\,\zA^j $. 
	We have $\dim \zA^j =1$ if $j \neq 0$ mod $2n+1$ and $\dim \zA^j=0$ otherwise.
	
	Let $\gA$ be realized as $L(\slt,C)$ and written out as $(2n+1) \times (2n+1)$-matrices. For $m\in \Z$ and $1 \leq j < 2n+1$ introduce the element
	\bea
	A_{(2n+1)m+j}=\xi^{(2n+1)m+j} \otimes \left(\begin{matrix}
		0 & I_{j} \\
		I_{2n+1-j} & 0 \\
	\end{matrix}\right) \quad \in \quad L(\slt,C),
	\eea
	where $I_{j}$ is the $j \times j$ identity matrix. We have $A_{(2n+1)m+j} = (A_1)^{(2n+1)m+j}$.
	
	If $\gA$ is realized as $L(\slt,\si_{0})$, introduce the element
	\bea
	B_{(2n+1)m+j}= \left(\begin{matrix}
		0 & \la^{m+1} \otimes I_{j} \\
		\la^{m}\otimes I_{2n+1-j} & 0 \\
	\end{matrix}\right) \quad \in\quad L(\slt,\si_{0}).
	\eea
	We have $B_{(2n+1)m+j} = (B_1)^{(2n+1)m+j}$.

	\begin{lem}
		For any $m\in \Z$, $1 \leq j < 2n+1$, the elements $		(\tau_C)^{-1}(A_{(2n+1)m+j})$,
\\		$		(\tau_0)^{-1}(B_{(2n+1)m+j})$  		of $\zA^{(2n+1)m+j}$
		are equal.
		\qed
	\end{lem}
	
Denote by $\La^{(1)}_{(2n+1)m+j}$ the  elements $(\tau_C)^{-1}(A_{(2n+1)m+j})$ and $(\tau_0)^{-1}(B_{(2n+1)m+j})$ of $\gA$.
	Notice that $\La^{(1)}_1=\sum_{i=0}^{2n}{E_{i}}=\La^{(1)}$.
	For any $m\in\Z,1\leq j< 2n+1,$  the element  $\La^{(1)}_{(2n+1)m+j}$  generates $\zA^{(2n+1)m+j} $.
	
	\medskip
	Let $T = \sum_{j=-\infty}^m T_j$ be a formal series with $T_j \in  \gA^j$.
	Denote $T^+ = \sum_{j=0}^n T_j,$\ {}\  $ T^- = \sum_{j<0} T_j$.
	Let $\gA$ be realized as $ \slt[\la,\la^{-1}]$. Consider $\La^{(1)}= B_{1}$ as a
	$(2n+1)\times (2n+1)$ matrix depending on the parameter $\la$.
	By \cite[Lemma 3.4]{DS}, we may represent $T$ uniquely in the form $T = \sum_{j=-\infty}^k b_j\,(\La^{(1)})^j$,\  $b_j\in \Dia$, where
	$\Dia\subset\frak{gl}_{2n+1}$ is the space of diagonal $(2n+1)\times (2n+1)$ matrices.
	Denote $(T)_{\La^{(1)}}^+ = \sum_{j=0}^k b_j\,(\La^{(1)})^j,$\ {}\  $ (T)_{\La^{(1)}}^- = \sum_{j<0} b_j\,(\La^{(1)})^j$.
	
	\begin{lem}
		We have $(T)_{\La^{(1)}}^+= T^+$, $(T)_{\La^{(1)}}^-=T^-$, $b_0=T^0$.
	\end{lem}
	\begin{proof}
		The isomorphism $ \iota : \slt[\la,\la^{-1}] \to L(\frak{sl}_{2n+1}, C^{(1)})$ is given by the formula
		$\la^m\otimes e_{k,l}\mapsto \xi^{(2n+1)m+k-l}\ox e_{k,l}$.
		We have $\iota(b_0) = \iota(1\ox(b_0^1 e_{1,1}+ \dots +b_0^{2n+1} e_{2n+1,2n+1}))=
1\otimes (b_0^1 e_{1,1}+\dots +b_0^{2n+1} e_{2n+1,2n+1})\in \gA^0$,
		$\iota(b_1\La^{(1)}) = \iota((b_1^1 e_{1,1}+ b_1^2 e_{2,2}+\dots +b_1^{2n+1}
 e_{2n+1,2n+1})(e_{2,1}+\dots + e_{2n+1,2n}+\la e_{1,2n+1}))
		=\iota(b_1^1 \la e_{1,2n+1}+b_1^2 e_{2,1}+\dots +b_1^{2n+1} e_{2n+1,2n})=\xi\otimes (b_1^1 e_{1,2n+1}+ b_1^2 e_{2,1}+\dots +b_1^{2n+1} e_{2n+1,2n})\in\gA^1$,
		$\iota(b_{-1}(\La^{(1)})^{-1}) = \iota((b_{-1}^1 e_{1,1}+ b_{-1}^2 e_{2,2}+\dots+b_{-1}^{2n+1} e_{2n+1,2n+1})(e_{1,2}+\dots +e_{2n,2n+1}+\la^{-1} e_{2n+1,1}))
		=\iota(b_{-1}^1 e_{1,2}+ \dots + b_{-1}^{2n} e_{2n,2n+1}+b_{-1}^{2n+1} \la^{-1}e_{2n+1,1})=\xi^{-1}\otimes (b_{-1}^1 e_{1,2}+\dots + b_{-1}^{2n} e_{2n,2n+1}+b_{-1}^{2n+1} e_{2n+1,1})\in\gA^{-1}$.
		Similarly one checks that $\iota(b_j\,(\La^{(1)})^j)\in\gA^j$ for any $j$.
	\end{proof}

We have ${(\La^{(1)})}^{-1}= \sum_{i=1}^{2n}e_{i,i+1}+\la^{-1}e_{2n+1,1}$,
	\bea
	E_{0}= \La^{(1)}e_{2n+1,2n+1}, \qquad  E_{i} = \La^{(1)}e_{i,i},\qquad
	F_{0}=e_{2n+1,2n+1} {(\La^{(1)})}^{-1},\qquad 
F_{i} = e_{i,i}{(\La^{(1)})}^{-1},
	\eea
for $i = 1, \dots, 2n$.

	\begin{lem}
		\label{lem exp}
Consider the elements $F_0,F_{i}+F_{2n+1-i}, 2(F_n+F_{n+1})$ for $i = 1, \dots, n-1$ as  $(2n+1)\times (2n+1)$ matrices.
		Let $g \in \C$. Then
		\bean
		\label{formula exp}
		&&
		e^{gF_0} = 1 + ge_{2n+1,2n+1}(\La^{(1)})^{-1},
		\\
		&&
\notag
		e^{g(F_i+F_{2n+1-i})} = 1 + g(e_{i,i}+e_{2n+1-i,2n+1-i})(\La^{(1)})^{-1},
		\\
		&&
\notag
		e^{g2(F_n+F_{n+1})} = 1 + 2g(e_{n,n}+e_{n+1,n+1})(\La^{(1)})^{-1} + 4g^2 e_{n,n} (\La^{(1)})^{-2}.
		\eean
		\qed
	\end{lem}

	\begin{lem}
		\label{lem lambda}  We have
		\bean
		\label{formula La}
		e_{i+1,i+1}\La^{(1)} = \La^{(1)} e_{i,i}, \qquad
		e_{i,i}(\La^{(1)})^{-1} = (\La^{(1)})^{-1}e_{i+1,i+1},
		\eean
		for all $i$, where we set $e_{2n+2,2n+2}=e_{1,1}$.
		\qed
	\end{lem}

	\section{Kac-Moody algebra of type $A_{2n}^{(2)}$}
\label{sec KM2}

	In this section we follow \cite[Section 5]{DS}.

	\subsection{Definition}

	For $n\geq 2$, consider the $(n+1) \times (n+1) $ Cartan matrix of type $\AT$,
	\bea
	\AT= \left(
	\begin{matrix}
		a_{0,0} & a_{0,1}  & \dots & a_{0,n} \\
		\dots  & \dots  & \dots  & \dots \\
		\dots  & \dots  & \dots  & \dots \\
		a_{n,0}  & a_{2n,1}  & \dots & a_{n,n}
	\end{matrix}
	\right)=\left(
	\begin{matrix}
		2 & -1  & 0 & \dots & \dots & \dots & 0\\
		-2  & 2  & -1  & 0 & \dots & \dots &\dots &\\
		0 & -1  & 2 & -1 & \dots & \dots & \dots \\
		\dots & 0 & -1 & \dots & \dots & \dots & \dots \\
		\dots & \dots & \dots & \dots &2 & -1 & 0 \\
		\dots & \dots & \dots & \dots & -1 & 2 & -1 \\
		0 & \dots & \dots & \dots & 0 & -2 & 2
	\end{matrix}
	\right).
	\eea
		The Kac-Moody algebra $\g(\AT)$ {\it of type} $\AT$ is the Lie algebra with {\it canonical generators} $e_i,h_i,f_i\in\gAA, i=0,\dots,n$,
	subject to the relations
	\bea
	&&
	[e_i,f_j]=\delta_{i,j}h_i,
	\qquad
	[h_i, e_j] = a_{i,j} e_j ,
	\qquad
	[h_i,  f_j] = - a_{i,j} f_j,
	\\
	&&
	(\on{ad} e_i)^{1-a_{i,j}} e_j=0,
	\qquad
	(\on{ad}  f_i)^{1-a_{i,j}} f_j=0,
	\qquad
	[h_i,h_j] = 0,
	\\
	&&
	2(h_{0}+h_{1}+\dots+h_{n-1})+h_{n}=0,
	\eea
see \cite[Section 5]{DS}.
	
	The Lie algebra $\gAA$ is graded with respect to the {\it standard grading}, $\deg e_i=1, \deg f_i=-1$,\ $ i=0,\dots,n $.
	Let $\gAA^j=\{x\in\gAA\ |\ \deg x=j\}$, then $\gAA = \oplus_{j\in\Z}\,\gAA^j$.

	Notice that $\gAA^0$ is the n-dimensional space generated by the $h_i$.
	Denote $\h=\gAA^0$. Introduce elements $\al_j$ of the dual space $\h^*$ by the conditions
	$\langle \al_j, h_i\rangle =a_{i,j}$ for $i,j=0,\dots,n$.

	\subsection{Realizations of $\gAA$}
	\label{R gAA}

	Consider the complex Lie algebra $\frak{sl}_{2n+1}$ with standard basis $e_{i,j}$,\ $i,j=1,2n+1$.
	
	Let $w=e^{2\pi i/ (4n+2)}$.
	Define the {\it Coxeter automorphism} $C : \frak{sl}_{2n+1}\to \frak{sl}_{2n+1} $  of order $4n+2$ by the formula
	\bea
	C(X) = -SX^{T}S^{-1},\ \ S = \on{diag}(1,-w,w^{2},\dots,w^{2n-2},-w^{2n-1}, w^{2n}),
	\eea
	where the $\{ \} ^T$ denotes transposition across the antidiagonal.
	Denote $(\frak{sl}_{2n+1})_j=\{ x \in \frak{sl}_{2n+1}\ | \ Cx=w^jx\}$.
	The twisted Lie subalgebra $L(\frak{sl}_{2n+1}, C)\subset \slt[\xi,\xi^{-1}]$ is the subalgebra
	\bea
	L(\frak{sl}_{2n+1}, C) = \oplus_{j \in \mathbb{Z}}\, \xi^{j} \otimes (\frak{sl}_{2n+1})_{j\, \on{mod}\, 4n +2} .
	\eea
	The isomorphism $\tau_C: \gAA \to L(\frak{sl}_{2n+1}, C)$ is defined by the formula
	\bea
	&&
	e_0\mapsto \xi\otimes e_{1,2n+1},
	\qquad
	e_n\mapsto \xi\otimes (e_{n+1,n}+e_{n+2,n+1})
	\\
	&&
	f_0\mapsto \xi^{-1}\otimes e_{2n+1,1},
	\qquad
	f_n\mapsto \xi^{-1}\otimes 2(e_{n,n+1}+e_{n+1,n+2})
	\\
	&&
	e_i\mapsto \xi\otimes (e_{i+1,i}+e_{2n+2-i,2n+1-i}),
	\qquad
	f_i\mapsto \xi^{-1}\otimes (e_{i,i+1}+e_{2n+1-i,2n+2-i}),
	\\
	&&
	h_0\mapsto 1\otimes (e_{1,1}-e_{2n+1,2n+1}),
	\qquad
	h_n\mapsto 1\otimes 2(-e_{n,n}+e_{n+2,n+2}),
	\\
	&&
	h_i\mapsto 1\otimes (-e_{i,i}+e_{i+1,i+1}-e_{2n+1-i,2n+1-i}+e_{2n+2-i,2n+2-i}).
	\eea
	Under this isomorphism we have $\gAA^j=\xi^j\otimes (\slt)_j$.
	Define the {\it standard automorphism} $\sigma_0 :\frak{sl}_{2n+1}\to \frak{sl}_{2n+1} $  of order 2 by the formula
	\bea
	\sigma_0(X) = -QX^{T}Q^{-1} = -X^{T},\ \ Q = \on{diag}(1,-1,\dots,-1,1).
	\eea
	Where the $\{ \} ^T$ is again transposition across the antidiagonal.
	Let $(\slt)_{0,j}=\{ x\in \slt\ | \ \si_0x=(-1)^jx\}$. Then the twisted Lie subalgebra $L(\frak{sl}_{2n+1}, \si_0)\subset \slt[\la,\la^{-1}]$ is the subalgebra
	\bea
	L(\frak{sl}_{2n+1}, \si_0) = \oplus_{j \in \mathbb{Z}} \,\la^{j} \otimes (\frak{sl}_{2n+1})_{0,\,j\, \on{mod}\, 2} .
	\eea
	The isomorphism $\tau_0: \gAA \to L(\frak{sl}_{2n+1}, \si_0)$ is defined by the formula
	\bea
	&&
	e_0\mapsto \la \otimes e_{1,2n+1},
	\qquad
	e_i\mapsto 1 \otimes (e_{i+1,i}+e_{2n+2-i,2n+1-i}),
	\\
	&&
	f_0\mapsto \lambda^{-1}\otimes e_{2n+1,1},
	\qquad
	f_i\mapsto1 \otimes (e_{i,i+1}+e_{2n+1-i,2n+2-i}),
	\\
	&&
	f_n \mapsto  1\otimes 2(e_{n,n+1}+e_{n+1,n+2}),
	\\
	&&
	h_0\mapsto 1\otimes (e_{1,1}-e_{2n+1,2n+1}),
	\qquad
	h_n \mapsto  1\otimes 2(-e_{n,n}+e_{n+2,n+2}),
	\\
	&&
	h_i\mapsto 1\otimes (-e_{i,i}+e_{i+1,i+1}-e_{2n+1-i,2n+1-i}+e_{2n+2-i,2n+2-i}).
	\eea
	
	\subsection{Element $\La^{(2)}$}

	Denote by $\La^{(2)}$ the element $\sum_{i=0}^{n}{e_{i}}\in \gAA$. Then $\zAA=\{x\in\gAA\ | \ [\La^{(2)},x]=0\}$ is an Abelian Lie subalgebra of $\gAA$. Denote $\z^j (A^{(2)}_{2n}) = \z (A^{(2)}_{2n})\cap \gAA^j$, then
	$\zAA=\oplus_{j\in\Z}\,\zAA^j$. 
	We have $\dim \zAA^j=0$ if j is even or $j = 2n+1 \mod 4n+2$ , and $\dim \zAA^{j}= 1$ otherwise.
	
For example, $\dim \z^j(A^{(2)}_{2}) = 1$ if  $j=1,5 \mod 6$ and
	the dimension equals zero otherwise; $\dim \z^j(A^{(2)}_{4}) = 1$ if  $j=1,3,7,9 \mod 10$ and
	the dimension equals zero otherwise.

	If $\gAA$ is realized as $L(\slt,C)$ and written out as $(2n+1) \times (2n+1)$ matrices, then for odd j, $1 \leq j < 2n+1$, introduce the element:
	\bea
	A_{(4n+2)m+j}=\xi^{(4n+2)m+j} \otimes \left(\begin{matrix}
		0 & I_{j} \\
		I_{2n+1-j} & 0 \\
	\end{matrix}\right)
	\qquad
	A_{(4n+2)m-j}=\xi^{(4n+2)m-j} \otimes \left(\begin{matrix}
				0 & I_{2n+1-j} \\
		I_{j} & 0 \\		
	\end{matrix}\right),
	\eea
	where $I_{j}$ is the $j \times j$ identity matrix. We have $A_{(4n+2)m\pm j} = A_{1}^{(4n+2)m \pm j}$.
	\bigskip
	If $\gAA$ is realized as $L(\slt,\si_0)$ and written out as $(2n+1) \times (2n+1)$ matrices, 
	then for odd j, $1 \leq j < 2n+1$, introduce the elements:
	\bea
	B_{(4n+2)m+j}=\left(\begin{matrix}
		0 & \la^{2m+1} \otimes I_{j} \\
		\la^{2m} \otimes I_{2n+1-j} & 0 \\
		
	\end{matrix}\right),
	\qquad
	B_{(4n+2)m-j}=\left(\begin{matrix}	
		0 & \la^{2m} \otimes I_{2n+1-j} \\
		\la^{2m-1}\otimes I_{j} & 0 \\
		
	\end{matrix}\right).
	\eea
	We have $B_{(4n+2)m\pm j} = B_{1}^{(4n+2)m \pm j}$.
	
	\begin{lem}
		For any $m\in \Z$, odd j, $1 \leq j < 2n+1$, the elements
		\bea
		(\tau_C)^{-1}(A_{(4n+2)m\pm j}),
		\qquad
		(\tau_0)^{-1}(B_{(4n+2)m\pm j}),
		\eea
		of $\zAA^{(4n+2)m\pm j}$
		are equal.
		\qed
	\end{lem}
	
	Denote the elements $(\tau_C)^{-1}(A_{(4n+2)m+j})$, and $(\tau_0)^{-1}(A_{(4n+2)m-j})$ of $\gAA$ by $\La^{(2)}_{(4n+2)m+j}$ and
	$\La^{(2)}_{(4n+2)m-j}$, respectively.  Notice that $\La^{(2)}_1=\sum_{i=0}^n{e_{i}}=\La^{(2)}$. We set $\La^{(2)}_j=0$ if j is even or $j = 2n +1 \mod 4n+2$.
	The element $\La^{(2)}_{(4n+2)m \pm j}$ generates $\zAA^{(4n+2)m \pm j}$.
	
	\subsection{Lie algebra $\gAA$ as a subalgebra of $\gA$}
	\label{sec sub}
	
	The map  $ \rho :\gAA \to \gA$,
	\bea
	&&
	e_{0}\mapsto E_{0},
	\qquad
	e_{i}\mapsto E_{i}+E_{2n+1-i},
	\qquad 
	e_{n}\mapsto E_{n}+E_{n+1},
	\\
	&&
	f_{0}\mapsto F_{0}, 
	\qquad
	f_{i}\mapsto F_{i}+F_{2n+1-i},
	\qquad
	f_{n} \mapsto 2(F_{n}+F_{n+1}),
	\\
	&&
	h_{0} \mapsto H_{0}, 
	\qquad
	h_{i}\mapsto H_{i}+H_{2n+1-i},
	\qquad
	h_{n}\mapsto 2(H_{n}+H_{n+1}),	
	\eea
	where  $i = 1,\dots, n-1$, realizes the Lie algebra $\gAA$ as a subalgebra of $\gA$.
	This embedding preserves the standard grading and $\rho (\La^{(2)})=\La^{(1)}$.
	We have $\rho (\zAA^{j}) \subset \zA^{j}$.

	\section{mKdV equations}
	\label{sec MKDV}

	In this section we follow \cite{DS}.

	\subsection{mKdV equations of type $A^{(1)}_{2n}$}

Denote by $\Bb$ the space of complex-valued functions of one variable $x$.  Given a finite dimensional vector space $W$, denote by $\Bb(W)$ the space of
$W$-valued functions of $x$. Denote by $\der$ the differential operator $\frac d{dx}$.

Consider the Lie algebra $\tilde{\g}(A_{2n}^{(1)})$ of the formal differential operators of the form $c\der + \sum_{i=-\infty}^k p_i$, $c\in \C, p_i\in \B(\gA^i)$.
Let $U=\sum_{i<0} U_i$, $U_i \in \B(\gA^{i})$. If $\L\in\tilde\g(A_{2n}^{(1)})$, define
\bea
e^{\ad U}(\L) = \L + [U,\L] + \frac 1{2!}[U,[U,\L]]+\dots \ .
\eea
The operator $e^{\ad U}(\L)$ belongs to $\tilde \g(A_{2n}^{(1)})$. The map $e^{\ad U}$ is an automorphism of the Lie algebra $\tilde \g(A_{2n}^{(1)})$.
The automorphisms of this type form a group.
If elements of $\gA$ are realized as matrices depending on a parameter
as in Section \ref{Realizations of gA}, then
$e^{\ad U}(\L) = e^U\L e^{-U}$.

	A {\it Miura oper} of type $A^{(1)}_{2n}$  is a differential operator of the form
	\bean
	\label{Miura la}
	\L = \der + \Lambda^{(1)} + V,
	\eean
	where $\Lambda^{(1)} = \sum_{i=0}^{2n} E_{i}\in\gA$ and $V\in \B(\gA^0)$.
	Any Miura oper of type $A^{(1)}_{2n}$ is an element of $\tilde \g(A^{(1)}_{2n})$.
	Denote by $\mc{M}(A^{(1)}_{2n})$ the space of all Miura opers of type $A^{(1)}_{2n}$.

	\begin{prop}
		[{\cite[Proposition 6.2]{DS}}]
		\label{Prop U3}
		
		For any Miura oper $\L$ of type $A^{(1)}_{2n}$ there exists an element
		$U=\sum_{i<0} U_i$, $U_i \in \B(\gA^{i})$, such that the operator $\L_0 = e^{\ad U}(\L)$ has the form
		\bea
		\L_0 = \der + \La^{(1)} + H,
		\eea
		where $H=\sum_{j<0} H_j, H_j\in \B(\zA^{j})$. If $U,\tilde U$ are two such elements, then
		$e^{\ad U}e^{-\ad \tilde U} = e^{\ad T}$, where $T = \sum_{j<0} T_j$, $T_j\in \z(A^{(1)}_{2n})^j$.
\qed	\end{prop}
	
	Let $\L, U$ be as in Proposition \ref{Prop U3}. Let  $r \neq 0$ mod $2n+1$.
	The element $\phi(\La^{(1)}_{r})=e^{-\ad U} (\La^{(1)}_{r})$ does not depend on the choice of $U$ in Proposition \ref{Prop U3}.
	
	The element $\phi(\La^{(1)}_{r})$
	is of the form $\sum^k_{i=-\infty} \phi(\La^{(1)}_{r})^i$, $\phi(\La^{(1)}_{r})^i\in\B(\gA^i)$.
	We set $\phi(\La^{(1)}_{r})^+ = \sum^k_{i=0} \phi(\La^{(1)}_{r})^i$,
	$\phi(\La^{(1)}_{r})^- = \sum_{i<0} \phi(\La^{(1)}_{r})^i$.

	Let $r\in\Z_{>0}$ and  $r \neq 0$ mod $2n+1$.
	The differential equation
	\bean
	\label{mKdVr 3}
	\frac{\partial \L}{\partial t_r}  = [\phi(\La^{(1)}_{r})^+, \L]
	\eean
	is called the {\it  $r$-th  mKdV equation} of type $A^{(1)}_{2n}$.
	
	Equation \Ref{mKdVr 3} defines vector fields $\frac\der{\der t_r}$ on the space $\M(A^{(1)}_{2n})$ of Miura opers of type $A^{(1)}_{2n}$. For
	all $r,s$, the vector fields $\frac\der{\der t_r}$, $\frac\der{\der t_s}$ commute, see \cite[Section 6]{DS}.
	
	\begin{lem}[{\cite{DS}}]
		\label{lem der1}
		
		We have
		\bean
		\label{mKdVr 0}
		\frac{\partial\L}{\partial t_r}  = - \frac d{dx}\phi(\La^{(1)}_{r})^0.
		\eean
\qed
	\end{lem}

	\subsection{mKdV equations of type $A^{(2)}_{2n}$}

	A {\it Miura oper} of type $A^{(2)}_{2n}$  is a differential operator of the form
	\bean
	\label{Miura la}
	\L = \der + \Lambda^{(2)} + V,
	\eean
	where $\Lambda^{(2)} =\sum_{i=0}^{n} e_{i} \in\gAA$ and $V\in \B(\gAA^0)$.
	Denote by $\mc{M}(A^{(2)}_{2n})$ the space of all Miura opers of type $A^{(2)}_{2n}$.

	\begin{prop}
		[{\cite[Proposition 6.2]{DS}}]
		\label{Prop U}
		
		For any Miura oper $\L$ of type $A^{(2)}_{2n}$ there exists an element
		$U=\sum_{i<0} U_i$, $U_i \in \B(\g(A^{(2)}_{2n})^{i})$, such that the operator $\L_0 = e^{\ad U}(\L)$ has the form
		\bea
		\L_0 = \der + \La^{(2)} + H,
		\eea
		where $H=\sum_{j<0} H_j, H_j\in \B(\zAA^{j})$. If $U,\tilde U$ are two such elements, then
		$e^{\ad U}e^{-\ad \tilde U} = e^{\ad T}$, where $T = \sum_{j<0} T_j$, $T_j\in \zAA^j$.
\qed	\end{prop}
	
	Let $\L, U$ be as in Proposition \ref{Prop U}. Let  $r$ be odd, $r\neq 2n+1 $ mod $4n+2$.
	The element $\phi(\La^{(2)}_r)=e^{-\ad U}(\La^{(2)}_r)$ does not depend on the choice of $U$ in Proposition \ref{Prop U}.
	
	The element $\phi(\La^{(2)}_r)$
	is of the form $\sum^k_{i=-\infty} \phi(\La^{(2)}_r)^i$, $\phi(\La^{(2)}_r)^i\in\B(\gAA^i)$.
	We set $\phi(\La^{(2)}_r)^+ = \sum^k_{i=0} \phi(\La^{(2)}_r)^i$,
	$\phi(\L^{(2)}a_r)^- = \sum_{i<0} \phi(\La^{(2)}_r)^i$.

	Let $r\in\Z_{>0}$, $r$ odd and  $r \neq 2n+1 $ mod $4n+2$.
	The differential equation
	\bean
	\label{mKdVr}
	\frac{\partial \L}{\partial t_r}  = [\phi(\La^{(2)}_r)^+, \L]
	\eean
	is called the {\it  $r$-th  mKdV equation} of type $A^{(2)}_{2n}$.
	
	Equation \Ref{mKdVr} defines vector fields $\frac\der{\der t_r}$ on the space $\M(A^{(2)}_2)$ of Miura opers. For
	all $r,s$, the vector fields $\frac\der{\der t_r}$, $\frac\der{\der t_s}$ commute, see \cite[Section 6]{DS}.

	\begin{lem}[{\cite{DS}}]
		\label{lem der2}
		
		We have
		\bean
		\label{mKdVr 0}
		\frac{\partial\L}{\partial t_r}  = - \frac d{dx}\phi(\La^{(2)}_r)^0.
		\eean
\qed
	\end{lem}

	\subsection{Comparison of  mKdV equations of types $A^{(2)}_{2n}$ and $A^{(1)}_{2n}$}
	\label{sec comp}

	Consider $\g(A^{(2)}_{2n})$ as a Lie subalgebra of $\gA$, see Section \ref{sec sub}. If $\L$ is a Miura oper
	of type $A^{(2)}_{2n}$, then it  is also a Miura oper of type $A^{(1)}_{2n}$.  We have $\mc M(A^{(2)}_{2n})\subset
\mc M(A^{(1)}_{2n})$,
\bean
\label{M1}
&&
\phantom{aaa}
\mc M(A^{(1)}_{2n}) = \{\L=\der + \La^{(1)} + \sum_{i=1}^{2n+1} v_ie_{i,i}\ |\  \sum_{i=1}^{2n+1} v_i =0\},
\\
&&
\notag
\mc M(A^{(2)}_{2n}) = \{\L=\der + \La^{(1)} + \sum_{i=1}^{2n+1} v_ie_{i,i}\ |  \sum_{i=1}^{2n+1} v_i =0,\
v_j+v_{2n+2-j}=0, j=1,\dots,2n+1\}.
\eean

	\begin{lem}
		Let $r$ be odd, $r \neq 2n+1$ mod $4n+2$, $r>0$. Let $\L^{A^{(2)}_{2n}}(t_r)$ be the solution of the $r$-th mKdV equation of type $A^{(2)}_{2n}$
		with initial condition $\L^{A^{(2)}_{2n}}(0)=\L$. Let $\L^{A^{(1)}_{2n}}(t_r)$ be the solution of the $r$-th mKdV equation of type $A^{(1)}_{2n}$
		with initial condition $\L^{A^{(1)}_{2n}}(0)=\L$. Then $\L^{A^{(2)}_{2n}}(t_r)=\L^{A^{(1)}_{2n}}(t_r)$ for all values of $t_r$.
		\qed
	\end{lem}
	
	\begin{proof}
		The element $U$ in Proposition \ref{Prop U} which is used to construct the mKdV equation of type $A^{(2)}_{2n}$ can be used also
		to construct the mKdV equation of type $A^{(1)}_{2n}$.
	\end{proof}

	\subsection{KdV equations of type $A^{(1)}_{2n}$}
	\label{sec KdV}

	Let $\Bb((\der^{-1}))$ be the algebra of formal pseudodifferential operators of the form
	$a=\sum_{i \in \Z} a_i \der^i$, with $a_i \in \Bb$ and finitely many terms with $i > 0$.
	The relations in this algebra are
	\bea
	\partial^k u - u \partial^k = \sum_{i = 1}^\infty k(k-1)\dots(k-i+1)\frac{d^i u}{dx^i}\partial^{k-i}
	\eea
	for any $k\in\Z$ and $u\in\Bb$.
	For $a = \sum_{i \in \Z} a_i \der^i \in \Bb((\der^{-1}))$, define $a^+ = \sum_{i \geq 0} a_i \der^i$.

	Denote $\Bb[\der] \subset \Bb((\der^{-1}))$ the subalgebra of differential operators $a = \sum_{i=0}^m a_i \der^i$
	with $m\in\Z_{\geq 0}$. Denote $\D \subset \Bb[\der]$ the affine subspace of differential operators of the
	form
\\
 $L=\der^{2n+1} + \sum\limits_{i=0}^{2n-1}u_i\der^{i}$.

	For $L \in \D$, there exists a unique $L^{\frac{1}{2n+1}} =\der + \sum_{i\leq 0}a_i\der^i \in \Bb((\der^{-1}))$
	such that  $(L^{\frac{1}{2n+1}})^{2n+1} = L$. For $r\in\N$, we have $L^{\frac{r}{2n+1}} = \der^r + \sum^{r-1}_{i=-\infty}
	b_i\der^i$, $b_i\in \B$.
	We set $(L^{\frac{r}{2n+1}})^+ = \der^r + \sum^{r-1}_{i=0} b_i\der^i$.

	For $r\in \N$, the differential equation
	\beq
	\label{KdVr}
	\frac{\partial L}{\partial t_r} = [L,(L^{\frac{r}{2n+1}})^+]
	\eeq
	is called the  {\it  $r$-th  KdV equation} of type $A^{(1)}_{2n}$.

	Equation \Ref{KdVr} defines flows $\frac{\partial}{\partial t_r}$ on the space $\D$.
	For all $r,s \in \N$ the flows $\frac{\partial}{\partial t_r}$ and $\frac{\partial}{\partial t_s}$ commute, see \cite{DS}.

	\subsection{Miura maps}
	\label{sec Miura maps}

	Let $\Ll = \der + \Lambda^{(1)} + V$ be a Miura oper of type $A^{(1)}_{2n}$ with $V = \sum_{k = 1}^{2n+1} v_k e_{k,k}$, $\sum_{k=1}^{2n+1}v_k=0$.  For $i=0,\dots,2n+1$, define
	the scalar differential operator $L_i=\der^{2n+1}+\sum_{j=0}^{2n-1}u_{j,i}\der^{j}\in\D$ by the formula:
	\bean
\label{miuramap}
	&
	L_0 =L_{2n+1}= (\der - v_{2n+1})(\der - v_{2n})\dots(\der - v_{2})(\der - v_1),
	\\
\notag
	&
	L_i = (\der - v_i)(\der - v_{i-1})\dots(\der - v_1)(\der - v_{2n+1})\dots(\der - v_{i+2})(\der - v_{i+1}),
	\eean
for $i=1,\dots,2n$.
	
	\begin{thm}[{\cite[Proposition 3.18]{DS}}]
		\label{thm mkdvtokdv}
		Let a Miura oper $\Ll$ satisfy the mKdV equation \Ref{mKdVr 3} for some $r$.  Then for every $i=0,\dots,2n$ the differential operator
		$L_i$ satisfies the KdV equation \Ref{KdVr}.
	\end{thm}
	
	For $i=0,\dots,2n+1$, we define the {\it  $i$-th Miura map} by the formula
	\bea
	\frak{ m}_i \ : \ \mc M(A^{(1)}_{2n})\  \to \ \D,
	\quad
	\Ll \ \mapsto \ L_i,
	\eea
	see \Ref{miuramap}.  
	
	For $i=0,1,\dots,2n$, an {\it $i$-oper } is a differential operator of the form
	\bea
	\Ll = \der + \Lambda^{(1)} + V + W,
	\eea
	with $V \in \Bb(\gA^0)$ and $W \in \Bb(\n^-_i)$.
	For $w \in \Bb(\n^-_i)$ and an $i$-oper $\Ll$, the differential operator
	$e^{\text{ad} \, w}(\L)$
	is an $i$-oper.  The $i$-opers $\Ll$ and $e^{\text{ad} \, w} (\Ll)$ are called {\it $i$-gauge equivalent.}
	A Miura oper is an $i$-oper for any $i$.

	\begin{thm}
		[{\cite[Proposition 3.10]{DS}}]
		\label{thm gaugemiura}
		If Miura opers $\Ll$ and $\tilde \Ll$ are $i$-gauge equivalent, then
		$\frak m_i(\Ll) = \frak m_i(\tilde\Ll)$.	
\qed

\end{thm}

	\section{Tangent maps to Miura maps}
\label{tmaps}

\subsection{Tangent spaces}
Consider the spaces of Miura opers $\mc M(A^{(2)}_{2n})\subset
\mc M(A^{(1)}_{2n})$. The tangent space to $\mc M(A^{(2)}_{2n})$ at a point $\L$ is
\bean
\label{TM}
{}
\\
\notag
&&
T_\L \mc M(A^{(2)}_{2n}) = \{X=\sum_{i=1}^{2n+1}X_ie_{i,i}\ |\
\sum_{i=1}^{2n+1}X_i =0, \ X_j+X_{2n+2-j}=0,\ j=1,\dots, 2n+1\},
\eean
where $X_i$ are functions of variable $x$.
	Recall $\D = \{  L=\der^{2n+1} + \sum\limits_{i=0}^{2n-1}u_i\der^{i}\}$.
The tangent space to $\D$ at a point $D$ is $T_D\D=\{Z= \sum_{i=0}^{2n-1}Z_i \der^{i}\}$,
where $Z_i$ are functions of $x$.

Consider the restrictions of Miura maps to $\mc M(A^{(2)}_{2n})$ and the corresponding tangent maps
\bean
\label{Tmap}
d\frak m_i : T_\L \mc M(A^{(2)}_{2n}) \to T_{\frak m_i(\L)} \D,  \quad i=1,\dots,2n+1.
\eean
By definition, if  $\L=\der +\La^{(1)}+\sum_{i=1}^{2n+1}v_ie_{i,i} \in \mc M(A^{(2)}_{2n})$, $X=\sum_{i=1}^{2n+1}X_ie_{i,i}\in T_\L \mc M(A^{(2)}_{2n})$, $d\frak m_i (X) = Z^i= \sum_{j=0}^{2n-1}Z_j^i \der^{j}$,
then
 \bean
	 \label{dif zero}
 Z^i &	& =  (-X_i)(\der-v_{i-1})\dots(\der-v_1)(\der-v_{2n+1})\dots(\der-v_{i+1})
\\
	 \notag
&&	 +\ (\der-v_i)(-X_{i-1})\dots(\der-v_1)(\der-v_{2n+1})\dots(\der-v_{i+1})+\dots
\\
	 \notag
&&	 +\ (\der-v_i)(\der-v_{i-1})\dots(-X_1)(\der-v_{2n+1})\dots(\der-v_{i+1})
\\
	 \notag
&&	 +\ (\der-v_i)(\der-v_{i-1})\dots(\der-v_1)(-X_{2n+1})\dots(\der-v_{i+1})+\dots
\\
	 \notag
&&	 +\ (\der-v_i)(\der-v_{i-1})\dots(\der-v_1)(\der-v_{2n+1})\dots(-X_{i+1}).
	 \eean

 In what follows we study the 
intersection of kernels of these tangent maps when $i$ runs through certain subsets of 
$\{1,\dots,2n+1\}$.
	
\subsection{Formula for the first  coefficient}

\begin{prop}
Let $\L=\der +\La^{(1)}+\sum_{i=1}^{2n+1}v_ie_{i,i} \in \mc M(A^{(1)}_{2n})$, $X=\sum_{i=1}^{2n+1}X_ie_{i,i}
$ $\in T_\L \mc M(A^{(2)}_{2n})$, $d\frak m_i (X) = Z^i= \sum_{j=0}^{2n-1}Z_j^i \der^{j}$. Then
\bean
\label{1st coeff}
Z_{2n-1}^i = -\left(\sum_{k=1}^{2n+1}v_kX_k+\sum_{k=1}^{i}(i-k)X'_{k} +\sum_{k=i+1}^{2n+1}(i+2n+1-k)X'_{k}
\right).
\eean
\end{prop}

\begin{proof} The proof uses only the identity $\sum_{j=1}^{2n+1}v_j=0$ and is straightforward.
\end{proof}

\subsection{Intersection of kernels of $d\frak m_i$}

\begin{lem}
\label {lem 6.10}

Let $\L=\der +\La^{(1)}+\sum_{k=1}^{2n+1}v_k e_{k,k} \in \mc M(A^{(2)}_{2n})$, $X=\sum_{k=1}^{2n+1}X_ke_{k,k}
$ $\in T_\L \mc M(A^{(2)}_{2n})$, $d\frak m_i (X) = Z^i= \sum_{j=0}^{2n-1}Z_j^i \der^{j}$. Assume that
$Z^i_{2n-1}=0$ for $i=1,\dots,2n$, then
\bean
\label{eqn j0}
X'_1-2v_{1}X_{1}=\sum_{k=2}^{2n}v_kX_k,
\qquad
X_i'=0, \quad i=2,\dots, 2n.
\eean
\end{lem}

	\begin{proof}
		By assumption we have the system of equations	
\bean
	\label{dif coeff}
&&
	X'_{2n-1}+2X'_{2n-2}+\dots+(2n-1)X'_{1}+2nX'_{2n+1}+\sum_{k=1}^{2n+1}v_kX_k=0,
\\
	\notag
&&
	X'_{2n-2}+2X'_{2n-3}+\dots+(2n-1)X'_{2n+1}+2nX'_{2n}+\sum_{k=1}^{2n+1}v_kX_k=0,\\
	\notag
&&
	X'_{2n-3}+2X'_{2n-4}+\dots+(2n-1)X'_{2n}+2nX'_{2n-1}+\sum_{k=1}^{2n+1}v_kX_k=0,\\
	\notag
&&
	\dots 
\\
	\notag
&&
	X'_{1}+2X'_{2n+1}+\dots+(2n-1)X'_{4}+2nX'_{3}+\sum_{k=1}^{2n+1}v_kX_k=0,\\
	\notag
&&
	X'_{2n+1}+\dots+(2n-1)X'_{3}+2nX'_{2}+\sum_{k=1}^{2n+1}v_kX_k=0.
	\eean
By subtracting the first equation from the second we get
$	2nX'_{2n}-X'_{2n-1}-X'_{2n-2}-\dots-X'_{1}-X'_{2n+1}=0$,
	equivalently
$	(2n+1)X'_{2n}-\sum_{k=1}^{2n+1}X'_k=0$.  
	Since $\sum_{k=1}^{2n+1}X_k=0$, we get $X'_{2n} = 0$. By subtracting the second from the third we get $X'_{2n-1}=0$. Similarly we obtain
	\bean
	\label{diff res0}
	X'_i = 0,
\quad i=2,\dots,2n.
	\eean
	Applying (\ref{diff res0}) to the last equation in (\ref{dif coeff}) yields 
$	X'_{2n+1}+\sum_{k=1}^{2n+1}v_kX_k=0$.
	By pulling out the terms for $k=1,2n+1$ we obtain
$X'_{2n+1}+v_1X_1+v_{2n+1}X_{2n+1}=-(X'_1-2v_{1}X_{1})=-\sum_{k=2}^{2n}v_kX_k$.
	\end{proof}

\begin{lem}
\label {lem 6.16}
Let $j\in \{1,\dots,n-1\}$.
Let $\L=\der +\La^{(1)}+\sum_{k=1}^{2n+1}v_k e_{k,k} 
\in \mc M(A^{(2)}_{2n})$, $X=\sum_{k=1}^{2n+1}X_ke_{k,k}
$ $\in T_\L \mc M(A^{(2)}_{2n})$, $d\frak m_i (X) = Z^i= \sum_{j=0}^{2n-1}Z_j^i \der^{j}$. Assume that
$Z^i_{2n-1}=0$ for all $i\notin\{ j, 2n+1-j\}$, then
\bea
					X'_{j}+v_{j}X_{j}+v_{j+1}X_{j+1}
					=-\sum_{k=1,k\neq j,j+1}^{n}v_kX_k,
\qquad  X'_j+X'_{j+1}=0,
\qquad  X'_i =0
					\eea
for $i\notin \{j, j+1, 2n+1-j, 2n+2-j\}$.
\end{lem}

\begin{proof}
By assumption we have the system of equations
			\bea
&&			X'_{2n}+2X'_{2n-1}+\dots+(2n-1)X'_{2}+2nX'_{1}+\sum_{k=1}^{2n+1}v_kX_k=0,\\
			\notag
	&&		X'_{2n-1}+2X'_{2n-2}+\dots+(2n-1)X'_{1}+2nX'_{2n+1}+\sum_{k=1}^{2n+1}v_kX_k=0,\\
			\notag
		&&	\dots \\
		&&	\notag
			X'_{2n+1-j}+\dots+(2n+1-j)X'_{1}+(2n+2-j)X'_{2n+1}+\dots
\\
&&
\phantom{aaaaaaaaaaaaaaaaaaaaaaaaaaaaaaaaaaaa}
\dots +2nX'_{2n+3-j}+\sum_{k=1}^{2n+1}v_kX_k=0,\\
		&&	\notag
			X'_{2n-1-j}+\dots+(2n-1-j)X'_{1}+(2n-j)X'_{2n+1}+\dots+2nX'_{2n+1-j}+\sum_{k=1}^{2n+1}v_kX_k=0,\\
		\notag
	&&		\dots \\
		&&	\notag
			X'_{j}+\dots+j X'_{1}+(j+1)X'_{2n+1}+\dots+2nX'_{j+2}+\sum_{k=1}^{2n+1}v_kX_k=0,\\
		&&	\dots 
			\eea
			Subtracting the second line from the first gives  $ X'_{2n+1}=0$,
cf. the proof of Lemma \ref{lem 6.10}.
Similarly, for $i\notin\{ j,j+1,2n+1-j,2n+2-j\}$  considering the difference $Z^{i-1}_{2n-1} - Z^i_{2n-1}=0$
we obtain $X'_i=0$.

Considering the difference  $Z^{2n+2-j}_{2n-1} - Z^{2n-j}_{2n-1}=0$ we obtain
 \bea
&&			X'_{2n+1-j}+\dots+(2n+1-j)X'_{1}+(2n+2-j)X'_{2n+1}+\dots+2nX'_{2n+3-j}
+\sum_{k=1}^{2n+1}v_kX_k
\\
&&			-\Big(X'_{2n-1-j}+\dots+(2n-1-j)X'_{1}+(2n-j)X'_{2n+1}+\dots+2nX'_{2n+1-j}
+\sum_{k=1}^{2n+1}v_kX_k\Big)
\\
	&&		=-(2n+1)(X'_{2n+1-j}+X'_{2n+2-j})+2\sum_{k=1}^{2n+1}X'_k=0.
			\eea
Hence $X'_{2n+1-j}+X'_{2n+2-j}=0$ and $X'_{j}+X'_{j+1}=0$.
			Now we can rewrite equation $Z^{2n+1}_{2n-1}=0$ as
	\bea
			(j-1)X'_{2n+2-j}+jX'_{2n+1-j}+(2n-j)X'_{j+1}+(2n+1-j)X'_{j}+\sum_{k=1}^{2n+1}v_kX_k=0.
			\eea
			Or equivalently
			\bea
			2X'_{j}+\sum_{k=1}^{2n+1}v_kX_k =
			2X'_{j}+2\sum_{k=1}^{n}v_kX_k = 
		2(	X'_{j}+v_{j}X_{j}+v_{j+1}X_{j+1}+\sum_{k=1,k\neq j,j+1}^{n}v_kX_k) =0.
			\eea
			\end{proof}

\begin{lem}
\label {lem 6.20}
Let $\L=\der +\La^{(1)}+\sum_{k=1}^{2n+1}v_k e_{k,k} 
\in \mc M(A^{(2)}_{2n})$, $X=\sum_{k=1}^{2n+1}X_ke_{k,k}
$ $\in T_\L \mc M(A^{(2)}_{2n})$, $d\frak m_i (X) = Z^i= \sum_{j=0}^{2n-1}Z_j^i \der^{j}$. Assume that
$Z^i_{2n-1}=0$ for all $i\notin\{ n, n+1\}$, then
\bea
					X'_{n}+v_{n}X_{n}=-\sum_{k=1}^{n-1}v_kX_k,
\qquad  X'_i =0, \quad i\notin \{n, n+2\}.
					\eea
\end{lem}

\begin{proof}
By assumption we have the system of equations
		\bea
&&		X'_{2n}+2X'_{2n-1}+\dots+(2n-1)X'_{2}+(2n)X'_{1}+\sum_{k=1}^{2n+1}v_kX_k=0,\\
		\notag
	&&	X'_{2n-1}+2X'_{2n-2}+\dots+(2n-1)X'_{1}+(2n)X'_{2n+1}+\sum_{k=1}^{2n+1}v_kX_k=0,\\
		\notag
	&&	\dots \\
	&&	\notag
		X'_{n+1}+\dots+(n+1)X'_{1}+(n+2)X'_{2n+1}+\dots+2nX'_{n+3}+\sum_{k=1}^{2n+1}v_kX_k=0,\\
	&&	\notag
		X'_{n-2}+\dots+(n-2)X'_{1}+(n-1)X'_{2n+1}+\dots+2nX'_{n}+\sum_{k=1}^{2n+1}v_kX_k=0,\\
&&		\notag
		\dots \\
	&&	\notag
		X'_{2n+1}+2X'_{2n}+\dots+2nX'_{2}+\sum_{k=1}^{2n+1}v_kX_k=0.
		\eea
		Subtracting the second line from the first gives $X'_{2n+1}=0$, cf. the proof of Lemma  \ref{lem 6.10}. 
Similarly, for $i\notin\{ n, n+1, n+2 \}$  considering the difference $Z^{i-1}_{2n-1} - Z^i_{2n-1}=0$
we obtain $X'_i=0$. Notice that $X_{n+1}=0$ by assumption.

Now we can rewrite equation $Z^{2n+1}_{2n-1}=0$ as 
		\bea
&&		(n-1)X'_{n+2}+(n+1)X'_{n}+\sum_{k=1}^{2n+1}v_kX_k
=	
2(X'_{n}+v_{n}X_{n} +\sum_{k=1}^{n-1}v_kX_k)=0.
		\eea
	\end{proof}

	\section{Critical points of master functions and generation of tuples of polynomials}
	\label{sec gene}
	In this section we follow \cite{MV1}.
	For functions $f(x),g(x)$,  we denote
	\bea
	\Wr(f,g) = f(x)g'(x)-f'(x)g(x)
	\eea
	the Wronskian determinant, and $f'(x):=\frac{df}{dx}(x)$.

	\subsection{Master function}
\label{mM}
		Choose nonnegative integers $\bs{k} = (k_0,k_1,\dots, k_{n})$.
		Consider variables $ u = (u_i^{(j)})$, where $j = 0,1,\dots, n$ and $i = 1,\dots,k_j$.
		The {\it master function} $\Phi(u; \bs k)$
		is defined by the formula:
		\bean
\label{Master}
		& \Phi(u,\bs k) =
		-4 \sum_{i,i'}
		\ln (u^{(n-1)}_i-u^{(n)}_{i'})-
		2\sum_{j=0}^{n-2} \sum_{i,i'} 
		\ln (u^{(j)}_i-u^{(j+1)}_{i'})\\
\notag
		& +  8\sum_{i<i'} 
		\ln (u^{(n)}_i-u^{(n)}_{i'})+4\sum_{j=1}^{n-1} \sum_{i<i'} 
		\ln (u^{(j)}_i-u^{(j)}_{i'})
		+2\sum_{i<i'} \ln (u^{(0)}_i-u^{(0)}_{i'}).		
		\eean
	The product of symmetric groups
	$\Sigma_{\bs k}=\Si_{k_0}\times \Si_{k_1} \times \dots \times \Si_{k_{n}}$ acts on the set of variables
	by permuting the coordinates with the same upper index. The  function $\Phi$ is symmetric with respect to the $\Si_{\bs k}$-action.
	A point $u$ is a {\it critical point} if $d\Phi=0$ at $u$. In other words, $u$ is a critical point if and only if the following equations equal zero: 
	\bean
	\label{bethe eqn}
	&&
\phantom{aaaa}
\sum_{l=1}^{k_{1}}\frac{-2}{u_{j}^{(0)}-u_{l}^{(1)}} + 
	\sum_{s \neq j} \frac{2}{u_{j}^{(0)}-u_{s}^{(0)}}, \qquad j=1,\dots,k_0, 
	\\
	\notag
&&
	\sum_{l=1}^{k_{i-1}}\frac{-2}{u_{j}^{(i)}-u_{l}^{(i-1)}} + \sum_{l=1}^{k_{i+1}}\frac{-2}{u_{j}^{(i)}-u_{l}^{(i+1)}}+ 
	\sum_{s \neq j} \frac{4}{u_{j}^{(i)}-u_{s}^{(i)}}, \quad i = 1, \dots, n-2, \ j=1,\dots,k_i,
	\\
	\notag
&&	\sum_{l=1}^{k_{n-2}}\frac{-2}{u_{j}^{(n-1)}-u_{l}^{(n-2)}}+ \sum_{l=1}^{k_{n}}\frac{-4}{u_{j}^{(n-1)}-u_{l}^{(n)}} + 
	\sum_{s \neq j} \frac{4}{u_{j}^{(n-1)}-u_{s}^{(n-1)}}, \qquad j=1,\dots,k_{n-1},
	\\
	\notag
	&&
\sum_{l=1}^{k_{n-1}}\frac{-4}{u_{j}^{(n)}-u_{l}^{(n-1)}} + 
	\sum_{s \neq j} \frac{8}{u_{j}^{(n)}-u_{s}^{(n)}}, \qquad j=1,\dots,k_n. 
	\eean
	All the orbits have the same cardinality $\prod_{i=0}^n k_i!$\ .	
	We do not make distinction between critical points in the same orbit.

\begin{rem}
The definition of master functions can be found in \cite{SV}, see also \cite{MV1, MV2}.
The master functions $\Phi(u,\bs k)$ in \Ref{Master}
are associated with the Kac-Moody algebra with Cartan matrix of type
		\beq
\label{A}
		A=(a_{i,j})=\left(
		\begin{matrix}
			2 & -2  & 0 & 0 & \dots & \dots & 0\\
			-1  & 2  & -1  & 0 & \dots & \dots &\dots &\\
			0 & -1  & 2 & -1 & \dots & \dots & \dots \\
			0 & 0 & -1 & \dots & \dots & \dots & \dots \\
			\dots & \dots & \dots & \dots &2 & -1 & 0 \\
			\dots & \dots & \dots & \dots & -1 & 2 & -2 \\
			0 & \dots & \dots & \dots & 0 & -1 & 2
		\end{matrix}
		\right),
		\eeq
		 which is dual to the Cartan matrix $\AT$, see  this Langlands duality in \cite{MV1, MV2, VWW}.
\end{rem}
	
	\subsection{Polynomials representing critical points}
	\label{PRCP}
	
	Let $u = (u_i^{(j)})$ be a critical point of the master function
	$\Phi$.
	Introduce the $(n+1)$-tuple of polynomials $ y= (y_{0}(x),\dots,y_{n}(x))$,
	\bean\label{y}
	y_j(x)\ =\ \prod_{i=1}^{k_j}(x-u_i^{(j)}).
	\eean
	This tuple of polynomials defines a point in the direct product
	$(\C[x])^{n+1}$.
	We say that the tuple {\it represents the
		critical point}.

		Each polynomial of the tuple will be considered up to multiplication
	by a nonzero number.

It is convenient to think that the $(n+1)$-tuple $\bs{y}^\emptyset = (1, \dots, 1)$ of constant polynomials
	represents  in $(\C[x])^{n+1}$, the critical point of the master function with no variables.
	This corresponds to the case  $\bs k = (0, \dots, 0)$.
	
	We say that a given tuple $ y\in (\C[x])^{n+1}$ is {\it generic} if each polynomial $y_i(x)$ has no multiple roots
 and for $i=0,\dots,n-1$
	the polynomials  $y_i(x)$ and
	$y_{i+1}(x)$ have no common roots. If a tuple represents a critical point, then it is generic,
	see \Ref{bethe eqn}.  For example, the tuple  ${y}^\emptyset $ is generic.

	\subsection{Elementary generation}
	\label{Elementary generation}
An $(n+1)$-tuple is called {\it fertile} if  there exist polynomials $\tilde y_0,\dots, \tilde y_n\in (\C[x])^{n+1}$ such that
	\bean
	\label{Wr-Cr}
	\Wr(\tilde{y_j},y_j) = \prod_{i\ne j}y_i^{-a_{i,j}} , \qquad j=0,1,\dots, n,
	\eean
where $a_{i,j}$ are the entries of the Cartan matrix of type $A^{(2)}_{2n}$, that is,
\bean
\label{Eqn}
&&
	\Wr(\tilde{y_0},y_0) = y_{1}^{2},\qquad
	\Wr(\tilde{y_i},y_i) = y_{i-1}y_{i+1},\qquad i=1,\dots,n-2,
	\\
\notag
&&
\phantom{aaaa}
	\Wr(\tilde{y}_{n-1},y_{n-1}) = y_{n-2}y_{n}^{2},\qquad 
	\Wr(\tilde{y}_{n},y_n) = y_{n-1}. 
	\eean

	For  example, $y^\emptyset$ is fertile and $\tilde y_j=x+c_j,$ where the $c_j$ are arbitrary numbers.

Assume that an $(n+1)$-tuple of polynomials $y=(y_0,\dots,y_n)$ is fertile.
	Equations \Ref{Wr-Cr} give us first order inhomogeneous differential equations with respect to $\tilde y_i$.
	 The solutions are
	\bean
	\label{deG 0}
	&
	\tilde y_0 = y_0\int \frac{y_1^{2}}{y_0^2} dx + c_{0} y_0,\\
	\label{deG i}
	&
	\tilde y_i = y_i\int \frac{y_{i-1}y_{i+1}}{y_i^2} dx + c_{i} y_i, \qquad i=1,\dots,n-2,\\
	\label{deG n-1}
	&
	\tilde y_{n-1} = y_{n-1}\int \frac{{y_{n-2}y_{n}^{2}}}{y_{n-1}^2} dx + c_{n-1} y_{n-1},\\
	\label{deG n}
	&
	\tilde y_n = y_n\int \frac{y_{n-1}}{y_n^2} dx + c_{n} y_n,
	\eean
	where $c_0,\dots,c_n$ are arbitrary numbers.
	For each $i=0,\dots,n$, the tuple
	\bean
	\label{simple i}
\phantom{aaaa}
	y^{(i)}(x,c_{i}) = (y_0(x), \dots, y_{i-1}(x), \tilde y_i(x,c_{i}),y_{i+1}(x),\dots, y_{n}(x))
	\ \in  \ (\C[x])^{n+1} \
	\eean
	forms a one-parameter family.  This family  is called
	the {\it generation  of tuples  from $ y$ in the $i$-th direction}.
	A tuple of this family is called an  immediate descendant of $ y$ in the $i$-th direction.
	
	\begin{thm}
		[\cite{MV1}]
		\label{fertile cor}
		${}$
		
		\begin{enumerate}
			\item[(i)]
			A generic tuple $y = (y_0, \dots, y_n)$, $\deg y_i=k_i$,
			represents a critical point of the master function
$\Phi(u;\bs k)$
			if and only if $y$ is fertile.
			
			\item[(ii)] If $ y$ represents a critical point,
			then for any $c\in\C$ the tuple $y^{(j)}(x,c)$ $j=0,\dots,n$ is fertile.
			
			\item[(iii)]
			If $y$ is generic and fertile, then for almost all values of the parameter
			$c\in \C$ the tuple $y^{(j)}(x,c)$ is generic.
			The exceptions form a finite set in $\C$.

			\item[(iv)]
			
			Assume that a sequence ${y}_i, i = 1,2,\dots$, of fertile tuples
			has a limit ${y}_\infty$ in $(\C[x])^{n+1}$ as $i$ tends to infinity.
			\begin{enumerate}
				\item[(a)]
				Then the limiting tuple ${y}_\infty$ is fertile.
				\item[(b)]
				For $j = 0,\dots, n$, let ${y}_\infty^{(j)}$ be an immediate
				descendant of ${y}_\infty$.
				Then for all $j$ there exist immediate descendants
				${y}_i^{(j)}$ of ${y}_i$ such that ${y}_\infty^{(j)}$
				is the limit of ${y}_i^{(j)}$ as $i$ tends to infinity.
				
			\end{enumerate}
			\qed
			
		\end{enumerate}
	\end{thm}

	\subsection{Degree increasing generation}
	\label{Degree increasing generation}
	
	Let $y=(y_0,\dots,y_n)$ be a generic fertile $(n+1)$-tuple of polynomials.
 Define $k_j=\deg y_j$ for 	 $j=0,\dots, n$.
	
	The polynomial $\tilde y_0$ in \Ref{deG 0}
	is of degree $k_0$ or
	$\tilde k_0=2k_1+ 1 - k_0$. We say that the  generation $(y_0,\dots,y_n) \to (\tilde y_0,\dots,y_n)$ is
	{\it  degree increasing } in the $0$-th direction  if $\tilde k_0 > k_0$. In that
	case $\deg \tilde y_0=\tilde k_0$ for all $c$.	
	
	For $i=1,\dots, n-2$, the polynomial $\tilde y_i$ in \Ref{deG i}
	is of degree $k_i$ or
	$\tilde k_i= k_{i-1}+k_{i+1} + 1 - k_i$. We say that the  generation $(y_0,\dots, y_{i}, \dots, y_n) \to (y_0,\dots, \tilde y_i, \dots, y_{n})$ is
	{\it  degree increasing } in the $i$-th direction  if $\tilde k_i > k_i$. In that
	case $\deg \tilde y_i=\tilde k_i$ for all $c$.
	
	The polynomial $\tilde y_{n-1}$ in \Ref{deG n-1}
	is of degree $k_{n-1}$ or
	$\tilde{k}_{n-1}=k_{n-2}+2k_{n}+ 1 - k_{n-1}$. We say that the  generation $(y_0,\dots,y_{n-1},y_n) \to (y_0,\dots,\tilde y_{n-1},y_n)$ is
	{\it  degree increasing } in the $n-1$-st direction  if $\tilde k_{n-1} > k_{n-1}$. In that
	case $\deg \tilde y_{n-1}=\tilde k_{n-1}$ for all $c$.
	
	The polynomial $\tilde y_n$ in \Ref{deG n}
	is of degree $k_n$ or
	$\tilde k_n=k_{n-1}+ 1 - k_n$. We say that the  generation $(y_0,\dots,y_{n-1},y_n) \to ( y_0,\dots,y_{n-1},\tilde y_n)$ is
	{\it  degree increasing } in the $n$-th direction if $\tilde k_n > k_n$. In that
	case $\deg \tilde y_n=\tilde k_n$ for all $c$.

	For $i= 0,\dots,n$, if the generation is degree increasing in the $i$-th direction we normalize family
	\Ref{simple i} and construct a map
	$ Y_{y,i} : \C \to (\C[x])^{n+1}$ as follows. First we multiply the polynomials $y_0,\dots, y_n$ by numbers to make them monic.
	Then we choose a monic polynomial $ y_{i,0}$ satisfying the equation $\Wr({y_{i,0}},y_i) = \ep\, \prod_{j\ne i}y_j^{-a_{j,i}},$ for some
nonzero interger $\ep$,
	and such that the coefficient of $x^{k_i}$ in $y_{i,0}$ equals zero. Set
	\bean
	\label{tilde yi}
	\tilde y_i(x,c)=y_{i,0}(x) + cy_i(x),
	\eean
	and define
	\bean
	\label{normalized i}
\phantom{aaaa}
	Y_{y,i} \ :\ \C\ \to\ (\C[x])^{n+1}, \qquad c \mapsto\ y^{(i)}(x,c) = (y_0(x),\dots,\tilde y_i(x,c),\dots,y_{n}(x)).
	\eean
	The polynomials of this $(n+1)$-tuple are monic.
	
	\subsection{Degree-transformations and generation of vectors of integers}
	\label{sec degree transf}

	The degree-transformations
	\bean
	\label{l-transformation}
	\phantom{aaa}
&&
	\bs k:=(k_0,\dots, k_n)\ \mapsto \
	\bs k^{(0)}=(2k_1+1-k_0,\dots, k_n),
	\\
	\notag
	&&
	\bs k:=(k_0,\dots,k_{n})\ \mapsto \
	\bs k^{(i)}=(k_0,\dots, k_{i-1}+k_{i+1}+1-k_i,\dots, k_{n}), 
\\
\notag
&&
\phantom{aaaaaaaaaaaaaaaaaaaaaaaaaaaaaaaaaaaaaaaaaaaaaaaaaaa}
\qquad i=1,\dots,n-2,
	\\
	\notag
	&&
	\bs k:=(k_0,\dots, k_{n-1},k_{n})\ \mapsto \
	\bs k^{(n-1)}=(k_0,\dots, k_{n-2}+2k_n+1-k_{n-1}, k_n),
	\\
	\notag
	&&
	\bs k:=(k_0,\dots,k_{n})\ \mapsto \
	\bs k^{(n)}=(k_0,\dots, k_{n-1}+1-k_n),
	\eean
	correspond to the shifted action of
  reflections $ w_0,\dots, w_n\in W$,
	where $W$ is the Weyl group associated with the Cartan matrix $A$ in \Ref{A}
 and $w_0,\dots, w_n$
	are the standard generators, see \cite[Lemma 3.11]{MV1} for more detail.

	We take formula \Ref{l-transformation} as the definition of {\it degree-transformations}:
		\bean
		\label{s i l-transf}
		&&
		w_0\ :\bs k\ \mapsto \
		\bs k^{(0)}=(2k_1+1-k_0,\dots, k_n),
		\\
		\notag
		&&
		w_i\ :\bs k\ \mapsto \
		\bs k^{(i)}=(k_0,\dots, k_{i-1}+k_{i+1}+1-k_i,\dots, k_{n}),
 \qquad i=1,\dots,n-2,
		\\
		\notag
		&&
		w_{n-1}\ :\bs k\ \mapsto \
		\bs k^{(n-1)}=(k_0,\dots, k_{n-2}+2k_n+1-k_{n-1}, k_n),
		\\
		\notag
		&&
		w_n\ :\bs k\ \mapsto \
		\bs k^{(n)}=(k_0,\dots, k_{n-1}+1-k_n),
		\eean
	acting on arbitrary vectors $\bs k=(k_0,\dots,k_n)$.

	We start with the vector $\bs k^\emptyset=(0,\dots,0)$ and a sequence $J=(j_1,j_2,\dots, j_m)$ of
	integers such that $j_i\in \{0,\dots,n\}$ for all $i$.
	We apply the corresponding degree transformations to  $\bs k^\emptyset$ and obtain
	a sequence of vectors $\bs k^\emptyset,$\ $  \bs k^{(j_1)} =w_{j_1}\bs k^\emptyset, $\ $
	\bs k^{(j_1,j_2)} = w_{j_2} w_{j_1}\bs k^\emptyset$,\dots,
	\bean
	\label{gen vector}
	\bs k^J  = w_{j_m}\dots w_{j_2} w_{j_1}\bs k^\emptyset .
	\eean
	We say that the {\it vector $\bs k^J $ is generated from $(0,\dots,0)$ in the direction of $J$}.
	
	We call a sequence $J$ {\it degree increasing} if for every $i$ the transformation
	$w_{j_i}$ applied to  $  w_{j_{i-1}}\dots w_{j_1}\bs k^\emptyset$
	increases the $j_i$-th coordinate.

	\subsection{Multistep generation}
	\label{sec generation procedure}
	
	Let $J = (j_1,\dots,j_m)$ be a degree increasing sequence.
	Starting from $y^\emptyset=(1,\dots,1)$ and $J$, we construct
	a map
	\bea
	Y^J : \C^m \to (\C[x])^{n+1}
	\eea
 by induction on $m$.
	If $J=\emptyset$, the map $Y^\emptyset$ is the map $\C^0=(pt)\ \mapsto y^\emptyset$.
	If $m=1$ and $J=(j_1)$,  the map
	$Y^{(j_1)} :  \C \to (\C[x])^{n+1}$ is given by formula \Ref{normalized i}
	for $y=y^\emptyset$ and $j=j_1$. More precisely, equation \Ref{Wr-Cr} 
	takes the form $\Wr(\tilde y_{j_{1}},1) =1$. Then $\tilde y_{j_{1},0}= x$ and
	\bea
	Y^{(j_{1})}\ :\ \C \mapsto (\C[x])^{n+1}, \qquad
	c \mapsto (1,\dots, x+c,\dots, 1).
	\eea
	By Theorem \ref{fertile cor}
	all tuples in the image are fertile and almost all tuples are generic
	(in this example all tuples are generic). 
	Assume that for ${\tilde J} = (j_1,\dots,j_{m-1})$,  the map
	$Y^{{\tilde J}}$ is constructed. To obtain  $Y^J$ we apply the
	generation procedure in the $j_m$-th
	direction to every tuple of the image of $Y^{{\tilde J}}$. More precisely, if
	\bean
	\label{J'}
	Y^{{\tilde J}}\ : \
	{\tilde c}=(c_1,\dots,c_{m-1}) \ \mapsto \ (y_0(x,{\tilde c}),\dots, y_n(x,{\tilde c})),
	\eean
	then 
	\bean
	\label{Ja}
	&&
	Y^{J} :	({\tilde c},c_m) \mapsto
	(y_0(x,{\tilde c}), \dots, y_{j_{m},0}(x,{\tilde c}) + c_m y_{j_{m}}(x,{\tilde c}),\dots, y_n(x,{\tilde c})).
	\eean
	The map  $Y^J$ is called  the {\it generation  of tuples   from $ y^\emptyset$ in the $J$-th direction}.

	\begin{lem}
		\label{lem gen procedure}
		All tuples in the image of $Y^J$ are fertile and almost all tuples are generic. For any $c\in\C^m$
		the $(n+1)$-tuple $Y^J(c)$ consists of monic polynomials. The degree vector of this tuple
		equals $\bs k^J$.
		\qed
	\end{lem}

	\begin{lem}
		\label{lem uniqeness}
		The map  $Y^J$ sends distinct points of $\C^m$ to distinct points of  $(\C[x])^{n+1}$.
	\end{lem}
	
	\begin{proof}
		The lemma is easily proved by induction on $m$.
	\end{proof}

\subsection{Critical points and the population generated from $y^\emptyset$}
	
	The set of all tuples $(y_0,\dots,y_n)\in (\C[x])^{n+1}$ obtained from $y^\emptyset=(1,\dots,1)$
	by generations in all  directions $J=(j_1,\dots,j_m)$, $m\ge 0$, (not necessarily degree increasing)
 is called the {\it population of tuples}
	generated from $y^\emptyset$, see \cite{MV1, MV2}.

\begin{thm} [{\cite{MV3}}]
\label{one gen}

If a tuple of polynomials $(y_0,\dots,y_n)$ represents a critical point of the master function
$\Phi(u,\bs k)$ defined in \Ref{Master} for some parameters $\bs k=(k_0,\dots,k_n)$, then 
$(y_0,\dots,y_n)$  is a point of the population generated from $y^\emptyset$ by a degree increasing generation, that is,
 there exist a degree increasing sequence $J=(j_1,\dots,j_m)$  and a point $c\in\C^m$
such that  $(y_0(x),\dots,y_n(x)) = Y^J(x,c)$.
Moreover,  for any other critical point of that function $\Phi(u,\bs k)$  there is a
point $c\rq{}\in\C^m$ such that
the tuple $ Y^J(x,c\rq{})$ represents that other critical point.

\end{thm}

By Theorem \ref{one gen} a function $\Phi(u,\bs k)$ either does not have critical points at
 all or all of its critical points form one cell $\C^m$.
	
\begin{proof}
Theorem 3.8 in \cite{MV2} says  that  
$(y_0,\dots,y_n)$  is a point of the population generated from $y^\emptyset$.  The fact that $(y_0,\dots,y_n)$  can be generated from 
$y^\emptyset$ by a degree increasing generation is a corollary of Lemmas 3.5 and 3.7 in \cite{MV2}.
The same lemmas  show that any other critical point of the master function $\Phi(u,\bs k)$ is represented by the tuple 
 $ Y^J(x,c\rq{})$  for a suitable $c\rq{}\in\C^m$.
\end{proof}

	\section{Critical points of master functions and Miura opers}
	\label{sec cr and Miu}

	\subsection{Miura oper associated with a tuple of polynomials, \cite{MV2}}
	We say that a Miura oper of type $\AT$, $\L=\der + \La^{(2)} + V$, is {\it associated to the $(n+1)$-tuple of polynomials} $y$ if
$	V=-\sum_{i=0}^{n}\ln'(y_i)\,h_i$, 
	where $\ln'(f(x))=\frac{f'(x)}{f(x)}$. If $\L$ is associated to $ y$ and
$V=\sum_{i=1}^{2n+1}v_i e_{i,i}$,  then
	\bean
	\label{v form}
\phantom{aaaaa}
	v_{i}=\ln'\frac{y_i}{y_{i-1}}=-v_{2n+2-i},
\quad i=1,\dots,n-1,
\quad
	v_n=\ln'\frac{y_n^2}{y_{n-1}}=-v_{n+2},\quad v_{n+1}=0.
	\eean
We also have
	\bean
	\label{Def eq}
	\langle \alpha_j, V\rangle = \ln' \big( \prod_{i=0}^n y_i^{-a_{i,j}} \big) ,
	\eean
where  $a_{i,j}$ are entries of the Cartan matrix of type $A^{(2)}_{2n}$.
	More precisely,
	\bean
	\label{def eq}
&&
	\langle \alpha_0, V\rangle = \ln' \big( \frac{y_1^{2}}{y_0^2}\big) ,
	\qquad
	\langle \alpha_i, V\rangle = \ln' \big( \frac{y_{i-1}y_{i+1}}{y_i^2} \big),\quad i=1,\dots,n-2,
	\\
	\notag
&&
\phantom{aaaaa}
	\langle \alpha_{n-1}, V\rangle = \ln' \big(  \frac{{y_{n-2}y_{n}^{2}}}{y_{n-1}^2}\big) ,
	\qquad
	\langle \alpha_n, V\rangle = \ln' \big( \frac{y_{n-1}}{y_n^2} \big).
	\eean
	For example,
	\bean
\label{Lem}
	\L^\emptyset: = \der +\La
	\eean
	is associated to the tuple $y^\emptyset=(1,\dots, 1)$.
	
	Define the map
	\bea
	\mu : (\C[x]\minus \{0\})^{n+1} \to \mc M(A^{(2)}_{2n}),
	\eea
	which sends a tuple $y=(y_0,\dots,y_n)$ to the Miura oper $\L=\der + \La^{(2)} + V$ associated to $y$.

	\subsection{Deformations of Miura opers of type $\AT$, \cite{MV2}}

	\begin{lem}[\cite{MV2}]
		Let $\L$ $= \der + \La^{(2)} + V$ be a Miura oper of type $\AT$.  Let $g \in \Bb$ and $j \in \{0,\dots, n\}$.  Then
		\bean \label{adeq}
		e^{\on{ad} \,g f_j} \L = \der + \La^{(2)} + V - g h_j - (g^\prime - \langle \alpha_j , V \rangle g + g^2)f_j.
		\eean
\qed
	\end{lem}

	\begin{cor}[\cite{MV2}]
		Let $\L = \der + \La^{(2)} + V$ be a Miura oper  of type $\AT$. Then $e^{\on{ad}\, g f_j } \L$
		is a Miura oper if and only if the scalar function
		$g$ satisfies the Riccati equation
		\bean
		\label{Ric}
		g' - \langle \alpha_j , V \rangle g +  g^2 = 0 \ .
		\eean
\qed
	\end{cor}
	
	Let $\L = \der + \La^{(2)} + V$ be a Miura oper.
	For $j\in\{0,\dots,n\}$, we say that $\L$ is  {\it deformable in the $j$-th direction}
	if equation \Ref{Ric} has a nonzero solution $g$, which is a rational function.

	\begin{thm} [\cite{MV2}]
		\label{ricc thm}
		Let
		$\L = \partial  +  \La^{(2)}  +  V$ be  the  Miura oper
associated to the tuple of polynomials $y=(y_0, \dots, y_n)$.
		Let $j\in \{0,\dots,n\}$. Then $\L$ is deformable in the $j$-th direction
		if and only if there exists a polynomial $\tilde y_j$ satisfying equation
		\Ref{Wr-Cr}. Moreover, in that case any nonzero rational
		solution $g$ of the Riccati equation \Ref{Ric} has the form
		$g = \mathrm{ln}' (\tilde y_j/ y_j)$ where $\tilde y_j$ is a solution of equation 
		\Ref{Wr-Cr}. If $g = \mathrm{ln}' (\tilde y_j/ y_j)$, then
		the Miura oper
		\bean
		\label{transformation}
		e^{\on{ad}\, g f_j} \L = \partial  + \La^{(2)} + V - g  h_j
		\eean
		is associated to the tuple $y^{(j)}$, which is obtained from the tuple $y$ by replacing $y_j$ with $\tilde y_j$.

	\end{thm}

	\subsection{Miura opers associated with the generation procedure}
	
	\label{Miura opers with cr points}

	Let $J = (j_1,\dots,j_m)$ be a degree increasing sequence, see Section \ref{sec degree transf}.
	Let $Y^J : \C^m \to (\C[x])^{n+1} $
	be the generation of tuples from $y^\emptyset$ in the $J$-th direction.
	We define the associated family of Miura opers by the formula:
	\bea
	\mu^J \  :\ \C^m\ \to\ \mc M(\AT), \qquad c \ \mapsto \ \mu(Y^J(c)) .
	\eea
	The map $\mu^J$ is called the {\it generation of Miura opers from $\Ll^\emptyset$
		in the $J$-th direction,}  see $\Ll^\emptyset$ in \Ref{Lem}.

	For $\ell=1,\dots,m$, denote $J_\ell=(j_1,\dots,j_\ell)$ the beginning $\ell$-interval of
	the sequence $J$. Consider the associated map $Y^{J_\ell} : \C^\ell\to(\C[x])^{n+1}$. Denote
	\bea
	Y^{J_\ell}(c_1,\dots,c_\ell) = (y_0(x,c_1,\dots,c_\ell, \ell),\dots, y_n(x,c_1,\dots,c_\ell, \ell)).
	\eea
	Introduce
	\bean
	\label{g's}
	g_1(x,c_1,\dots,c_m) &=&
	\ln'(y_{j_1}(x,c_1,1)) ,
	\\
	\notag
	g_\ell(x,c_1,\dots,c_m) &=&
	\ln'(y_{j_\ell}(x,c_1.\dots,c_\ell,\ell)) - \ln'(y_{j_\ell}(x,c_1,\dots,c_{\ell-1},\ell-1)),
	\eean
	for $\ell=2,\dots,m$.
	For $c\in\C^m$, define  $U^J(c)=\sum_{i<0} (U^J(c))_i$, $ (U^J(c))_i\in\B(\g(A^{(2)}_{2n})^i)$, depending on $c\in\C^m$,
	by the formula
	\bean
	\label{T}
	e^{-\,\ad U^J(c)} = e^{\ad g_m(x,c) f_{j_m}} \cdots e^{\ad g_1(x,c) f_{j_1}}.
	\eean

	\begin{lem}
		\label{lem formula} For $c\in\C^m$, we have
		\bean
		\label{operformula}
		\mu^J(c)\ &=& \ e^{-\,\ad U^J(c)}(\Ll^\8),
\\
		\label{oper2}
		\mu^J(c)\ &=&\
		\der + \La^{(2)} -\sum_{\ell=1}^m g_\ell(x,c) h_{j_\ell}.
		\eean
	\end{lem}
	
	\begin{proof}
		The lemma follows from Theorem \ref{ricc thm}.
	\end{proof}

	\begin{cor}
		\label{cor der}
		Let $r>0$, odd and $r\neq n+1$ mod $4n+2$. Let $c\in\C^m$. Let $\frac {\der\phantom{a}}{\der t_r}\big|_{\mu^J(c)}$ be the value at $\mu^J(c)$
		of the vector field of the
		$r$-th mKdV flow on the space $\mc M(A^{(2)}_{2n})$, see \Ref{mKdVr}. Then
		\bean
		\label{T mkdv}
		\frac {\der}{\der t_r}\Big|_{\mu^J(c)} = - \frac {\der}{\der x}\Big(e^{-\,\ad U^J(c)}(\La^{(2)})^r\Big)^0.
		\eean

	\end{cor}

	\begin{proof}
		The corollary follows from \Ref{mKdVr 0} and \Ref{operformula}.
	\end{proof}

	We have the natural  embedding $\mc M(\AT) \hookrightarrow \mc M(\At)$, see 
Section \ref{sec sub}. 
	Let $J=(j_1,j_2$, \dots, $j_m)$. Denote 
 ${\tilde J}=(j_1,\dots,j_{m-1})$. Consider the associated family
	$\mu^{{\tilde J}} : \C^{m-1}\to\mc M(\AT)$. Denote ${\tilde c}=(c_1,\dots,c_{m-1})$.

	\begin{prop}
\label{prop77}
	For any $r>0$  the difference
$	\frac {\der}{\der t_r}\big|_{\mu^J(c)} - \frac {\der}{\der t_r}\big|_{\mu^{\tilde{J}}(\tilde c)}$
has the following form for some scalar functions  $u_1(x,c)$, $u_2(x,c)$, $u_3(x,c)$:
\begin{enumerate}
\item[(i)]
if $j_m\in\{1,2,\dots,n-1\}$, then
	\bean
\label{k1}
	\frac {\der}{\der t_r}\big|_{\mu^J(c)} -\frac {\der}{\der t_r}\big|_{\mu^{\tilde{J}}(\tilde{c})}
&=&
 u_1(x,c)(e_{j_m+1,j_m+1}-e_{j_m,j_m})
\\
\notag
&+&
u_2(x,c)(e_{2n+2-j_m,2n+2-j_m}-e_{2n+1-j_m,2n+1-j_m}),
	\eean

\item[(ii)]  if $j_m=0$, then
\bean
\label{k2}
	\frac {\der}{\der t_r}\big|_{\mu^J(c)} -\frac {\der}{\der t_r}\big|_{\mu^{\tilde{J}}(\tilde{c})}
=
 u_1(x,c)(e_{2n+1,2n+1}-e_{1,1}),
\eean

\item[(iii)] if $j_m=n$, then
\bean
\label{k3}
	\frac {\der}{\der t_r}\big|_{\mu^J(c)} -\frac {\der}{\der t_r}\big|_{\mu^{\tilde{J}}(\tilde{c})}
=
 u_1(x,c)e_{n,n} +  u_2(x,c)e_{n+1,n+1}+  u_3(x,c)e_{n+2,n+2}.
\eean

\end{enumerate}
	\end{prop}

	\begin{proof}
We will write $\La$ for $\La^{(2)}=\La^{(1)}$. Denote
\bea
A_r
= e^{g_{m-1}f_{j_{m-1}}}\dots e^{g_{1}f_{j_{1}}} \La^r
e^{-g_1f_{j_1}}
\dots e^{-g_{m-1} f_{j_{m-1}}}.
\eea
Expand  $A_r = \sum_{i} A_r^i \La^i$ where $A_r^i = \sum_{l=1}^{2n+1}
A_r^{i,l}e_{l,l}$ with scalar coefficients $A_r^{i,l}$. 
	Then 
$\frac {\der}{\der t_r}\big|_{\mu^{\tilde{J}}(\tilde{c})}=-\frac{\der}{\der x} A_r^{0}$.
	Assume that $j_m\in\{1,\dots,n-1\}$. Then
	\bea
&&	\frac {\der}{\der t_r}\big|_{\mu^J(c)} =- \frac {\der}{\der x}\big[(1+g_m(e_{j_m,j_m}+e_{2n+1-j_m,2n+1-j_m})\La^{-1})A_r
\\
&&
\times
(1-g_m(e_{j_m,j_m}+e_{2n+1-j_m,2n+1-j_m})\La^{-1})\big]^0
	=- \frac {\der}{\der x}A_r^0
\\
&&	- \frac {\der}{\der x}\big[g_m(e_{j_m,j_m}+e_{2n+1-j_m,2n+1-j_m})\La^{-1}A_r\big]^0\\
&&	+\frac {\der}{\der x}\big[A_rg_m(e_{j_m,j_m}+e_{2n+1-j_m,2n+1-j_m})\La^{-1}\big]^0\\
&&	+ \frac {\der}{\der x}\big[g_m(e_{j_m,j_m}+e_{2n+1-j_m,2n+1-j_m})\La^{-1}A_rg_m(e_{j_m,j_m}+e_{2n+1-j_m,2n+1-j_m})\La^{-1}\big]^0.
	\eea
The last term is zero since
	\bea
	&&
\big[g_m(e_{j_m,j_m}+e_{2n+1-j_m,2n+1-j_m})\La^{-1}A_rg_m(e_{j_m,j_m}+e_{2n+1-j_m,2n+1-j_m})\La^{-1}\big]^0
\\
&&	=g_m^2(e_{j_m,j_m}+e_{2n+1-j_m,2n+1-j_m})[\La^{-1}A_r\La^{-1}]^0(e_{j_m+1,j_m+1}+e_{2n+2-j_m,2n+2-j_m})=0,
	\eea
and we get
	\bea
	\frac {\der}{\der t_r}\Big|_{\mu^J(c)} - \frac {\der}{\der t_r}\Big|_{\mu^{\tilde{J}}(\tilde{c})}
&=&
g_mA^{1,j_m+1}_r(e_{j_m+1,j_m+1} - e_{j_m,j_m})
\\
&+& 
g_mA^{1,2n+2-j_m}_r(e_{2n+2-j_m,2n+2-j_m} - e_{2n+1-j_m,2n+1-j_m}).
\eea
The cases  $j_m=0, n$ are proved similarly.
	\end{proof}

Let $\frak m_i : \mc M(\At)\to \D,\ \Ll \mapsto L_i,$ be the Miura maps
	defined in Section \ref{sec Miura maps} for  $i=0,\dots,2n$. Below we consider the composition of the embedding $\mc M(\AT) \hookrightarrow \mc M(\At)$
and a Miura map.

	\begin{lem}
		\label{lem j j'}
	If $j_m=0$,  we have  $\frak m_i\circ \mu^J({\tilde c},c_m) = \frak m_i\circ \mu^{{\tilde J}}({\tilde c})$ for all $i\neq 0$. If $j_m=1,\dots,n$,
we have $\frak m_i\circ \mu^J({\tilde c},c_m) = \frak m_i\circ \mu^{{\tilde J}}({\tilde c})$ for all $i\neq j_m, 2n+1-j_m$.
	\end{lem}
	
	\begin{proof}
		The lemma follows from formula \Ref{operformula} and Theorem \ref{thm gaugemiura}.
	\end{proof}

	\begin{lem}
		\label{lem der c-m}
		If $j_m=0$, then
		\bean
		\label{form der c-m 0}
		\frac{\der \mu^J}{\der c_m}({\tilde c},c_m) \ =\ -a\
		\frac{y_{1}(x,{\tilde c},m-1)^2}
		{y_{0}(x,{\tilde c},c_m,m)^2} \,h_{0}
		\eean
		for some  positive integer $a$. If $j_m=1,\dots,n-2$, then
		\bean
		\label{form der c-m i}
		\frac{\der \mu^J}{\der c_m}({\tilde c},c_m) \ =\ -a\
		\frac{y_{j_{m}-1}(x,{\tilde c},m-1)y_{j_{m}+1}(x,{\tilde c},m-1)}
		{y_{j_{m}}(x,{\tilde c},c_m,m)^2} \,h_{j_{m}}
		\eean
		for some  positive integer $a$. If $j_m=n-1$, then
		\bean
		\label{form der c-m n-1}
		\frac{\der \mu^J}{\der c_m}({\tilde c},c_m) \ =\ -a\
		\frac{y_{n-2}(x,{\tilde c},m-1)y_{n}^{2}(x,{\tilde c},m-1)}
		{y_{n-1}(x,{\tilde c},c_m,m)^2} \,h_{n-1}
		\eean
		for some  positive integer $a$.  If $j_m=n$, then
		\bean
		\label{form der c-m n}
		\frac{\der \mu^J}{\der c_m}({\tilde c},c_m) \ =\ -a\
		\frac{y_{n-1}(x,{\tilde c},m-1)}
		{y_{n}(x,{\tilde c},c_m,m)^2} \,h_{n}
		\eean
		for some  positive integer $a$.
	\end{lem}
	
Notice that the right-hand side of these formulas can be written as
\bean
\label{univ}
-a
\prod_{i=0}^n y_i(x,c,m)^{-a_{i,j}} h_j.
\eean

	\begin{proof}
		Let $j_m=0$. Then $ y_{0}(x,{\tilde c},c_m,m)=y_{0,0}(x,{\tilde c}) + c_m y_{0}(x,{\tilde c},m-1),$
		where $y_{0,0}(x,{\tilde c})$ is such that
		\bea
		\Wr ( y_{0,0}(x,{\tilde c}),y_{0}(x,{\tilde c},m-1)) \ =\ a\ y_{1}(x,{\tilde c},m-1)^2,
		\eea
		for some positive integer $a$, see \Ref{tilde yi}.
		We have $ g_m = \ln'(y_{0}(x,{\tilde c},c_m,m))- \ln' (y_{0}(x,{\tilde c},m-1))$.
		
		By formula \Ref{oper2}, we have
		\bea
&&		\frac{\der \mu^J}{\der c_m}({\tilde c},c_m) = -\frac{\der g_m}{\der c_m}({\tilde c},c_m) h_0
 = - \frac \der{\der c_m}\left(
		\frac{y_{0,0}'(x,{\tilde c}) + c_m y_{0}'(x,{\tilde c},m-1)}{y_{0,0}(x,{\tilde c}) + c_m y_{0}(x,{\tilde c},m-1)}
		\right) h_0=
		\\
		&&
		\phantom{aaaa}
		= 
	-	\frac{\Wr ( y_{0,0}(x,{\tilde c}), y_{0}(x,{\tilde c},m-1))}{(y_{0,0}(x,{\tilde c}) + c_m y_{0}(x,{\tilde c},m-1))^2} h_0
		=- a\
		\frac{y_{1}(x,{\tilde c},m-1)^2}
		{y_{0}(x,{\tilde c},c_m,m)^2} h_0\,.
		\eea
		This proves formula \Ref{form der c-m 0}. The other formulas are proved similarly.
	\end{proof}
	
	\subsection{Intersection of kernels of $d\frak m_i$}

Let $J = (j_1,\dots,j_m)$ be a degree increasing sequence and
$	\mu^J  : \C^m\to \mc M(\AT)$ 
 the generation of Miura opers from $\Ll^\emptyset$
		in the $J$-th direction. We have $\mu^J(c) = \der + \La^{(1)} + \sum_{k=1}^{2n+1} v_k(x,c) e_{k,k},$
where
\bea
 \sum_{k=1}^{2n+1} v_k(x,c) =0, \quad v_i(x,c)+v_{2n+2-i}(x,c)=0,\ 
i=1,\dots, 2n+1.
\eea
Let $X(c)=\sum_{k=1}^{2n+1} X_k(x,c) e_{k,k} \in T_{\mu^J(c)} \mc M(\AT)$ be
a field of tangent vectors to $\mc M(\AT)$ at the points of the image of $\mu^J$,
\bea
 \sum_{k=1}^{2n+1} X_k(x,c) =0, \quad X_i(x,c)+X_{2n+2-i}(x,c)=0,\ 
i=1,\dots, 2n+1.
\eea
Our goal is to show that under certain conditions we have
\bean
\label{X(c)0}
	X(c) = A(c) \frac{\der\mu^J}{\der c_m}(c)
\eean
 	for some scalar function $A(c)$ on $\C^{m}$.

\begin{prop} 
 	\label{prop 6.8}
 	Let $j_m=0$ and   $X(c)\in T_{\mu^J(c)} \mc M(\AT)$. Assume that
 $	d\frak{m}_{i}\big|_{\mu^J(c)}(X(c))=0$
 	for all $i = 1,\dots, 2n$ and all $c \in \C^{m}$.
Assume that $X(c)$ has the form indicated in the right-hand side of formula 
\Ref{k2}. Then equation \Ref{X(c)0} holds.
 \end{prop}

	\begin{proof}
Since $X_k(x,c)= 0$ for $k=2,\dots,2n$,  equation \Ref{eqn j0} takes the form
$X'_1-2v_{1}X_{1}=0$, or more precisely,
$		X'_{1}=2 \ln'\big(		\frac{y_{1}(x,{\tilde c},m-1)}
			{y_{0}(x,{\tilde c},c_m,m)}\big) X_{1}$.
Hence $X_1(x,c) = -X_{2n+1}= A(c) \frac{y_{1}(x,{\tilde c},m-1)^2}
			{y_{0}(x,{\tilde c},c_m,m)^2}$
			for some  scalar $A(c)$. Lemma \ref{lem der c-m} 
implies equation \Ref{X(c)0}.
	\end{proof}

			\begin{prop} 
				\label{prop 6.13}
				Let $j_m\in\{1,\dots,n-1\}$
 and   $X(c)\in T_{\mu^J(c)} \mc M(\AT)$. Assume that
$ 	d\frak{m}_{i}\big|_{\mu^J(c)}(X(c))=0$
 	for all $i \notin\{ j_m,2n+1-j_m\}$ and all $c \in \C^{m}$.
 Assume that $X(c)$ has the form indicated in the right-hand side of formula 
\Ref{k1}.  Then equation \Ref{X(c)0} holds.

\end{prop}

\begin{proof}		
By Lemma \ref{lem 6.16} we have
$	X'_{j_m}+(v_{j_m}-v_{j_m+1})X_{j_m}	=0$. Then for $j_m=1,\dots,n-2$, we have
	\bea
	X_{j_m}=-X_{j_m+1} = X_{2n+1-j_m}=-X_{2n+2-j_m} 
	=A(c)\
	\frac{y_{j_{m}-1}(x,{\tilde c},m-1)y_{j_{m}+1}(x,{\tilde c},m-1)}
	{y_{j_{m}}(x,{\tilde c},c_m,m)^2} 
	\eea
 and for $j_m=n-1$, we have
	\bea
	X_{n-1}=-X_n=X_{n+2}=-X_{n+3} 
	= A(c)\
	\frac{y_{n-2}(x,{\tilde c},m-1)y_{n}^{2}(x,{\tilde c},m-1)}
	{y_{n-1}(x,{\tilde c},c_m,m)^2} .
	\eea
Lemma \ref{lem der c-m} yields equation \Ref{X(c)0}.
\end{proof}

	\begin{prop} 
		\label{prop n}
		Let $j_m=n$ and   $X(c)\in T_{\mu^J(c)} \mc M(\AT)$. 
		 Assume that
$		d\frak{m}_{i}\big|_{\mu^J(c)}(X(c))=0$ 
		for all $i \neq n,n+1$, and $c \in \C^{m}$.
Assume that $X(c)$ has the form indicated in the right-hand side of formula 
\Ref{k3}. 		Then equation \Ref{X(c)0} holds.
	\end{prop}
	
	\begin{proof}
By assumptions we have  $X_{n+1}=0$ as well as $X_{i}=0$ for $i\ne n, n+1, n+2$.
By Lemma \ref{lem 6.20} we have $X_n\rq{}+v_nX_n=0$ where 
$v_n=\ln'\frac{y_n^2}{y_{n-1}}$.  Hence
$X_n=-X_{n+2}=A(c)\frac{y_{n-1}(x,{\tilde c},m-1)}
{y_{n}(x,{\tilde c},c_m,m)^2}$ for some scalar function $A(c)$.
Lemma \ref{lem der c-m} yields equation \Ref{X(c)0}.
\end{proof}

	\section{Vector fields}
	\label{sec Vector fields}

	\subsection{Statement}
	\label{sec statement}
	
	Let $r>0$ be odd, and $r\neq 2n+1$ mod $4n+2$. Recall that  we denote by $ \frac \der{\der t_r}$ the $r$-th mKdV vector field
	on the space  $\mc M(\AT)$  of Miura opers of type $\AT$. We also denote by $ \frac \der{\der t_r}$ the $r$-th mKdV vector field of type $\At$
	on the space  $\mc M(\At)$  of Miura opers of type $\At$.
	We have a natural embedding $\mc M(\AT) \hookrightarrow \mc M(\At)$.
	Under this embedding the vector $ \frac \der{\der t_r}$  on $\mc M(\AT)$ equals the vector filed $ \frac \der{\der t_r}$ on
	$\mc M(\At)$ restricted to  $\mc M(\AT)$,  see Section \ref{sec comp}.
	We also denote by  $ \frac \der{\der t_r}$ the $r$-th KdV vector field
	on the space  $\D$, see  Section \ref{sec KdV}.
	
	For a Miura map $\frak m_i : \mc M\to \D,\ \Ll \mapsto L_i,$ denote by $d\frak m_i$  the associated
	derivative map $T\mc M(\At) \to T\D$ of tangent spaces.  By Theorem
	\ref{thm mkdvtokdv} we have $d\frak m_i:  \frac \der{\der t_r}\big|_{\Ll}
	\mapsto \frac \der{\der t_r}\big|_{L_i}$.

	Fix a degree increasing sequence $J=(j_1,\dots, j_m)$. Consider the associated family $\mu^J:\C^m\to \mc M(\AT)$
	of Miura opers.
	For a vector field $\Ga $ on $\C^m$, we denote by  $\frak{L}_\Ga\mu^J$
	 the derivative of $\mu^J$ along the vector field.
	The derivative is well-defined since $\mc M(\AT)$ is an affine space.

	\begin{thm}
		\label{thm main}
		Let $r>0$ be odd, $r\neq 2n+1$ mod $4n+2$. Then there exists a polynomial vector field $\Ga_r$ on $\C^m$
		such that
		\bean
		\label{formula main}
		\frac \der{\der t_r}\Big|_{\mu^J(c)} =\frak{L}_{\Ga_r}\mu^J(c) 
		\eean
		for all $c\in\C^m$. If $r>4m$, then
$		\frac \der{\der t_r}\big|_{\mu^J(c)} =0$.
	\end{thm}
	
	\begin{cor}
		\label{cor Main}
		The family $\mu^J$ of Miura opers is invariant with respect to all mKdV flows of type $\AT$ and
		is point-wise fixed by flows with $r>4m$.
		
	\end{cor}
	
In other words, every mKdV flow corresponds to a flow on the space
of integration parameters $c\in\C^m$. 
Informally speaking, we may say, that the integration parameters
$c = (c_1,\dots,c_m)$ are times of the mKdV flows.

	\subsection{Proof of Theorem \ref{thm main} for $m=1$}
Let  $J=(j_1)$. Then 
$\mu^{J}(c_1) = e^{g_1f_{j_1}}\Ll^\emptyset e^{-g_1 f_{j_1}}=
\der +\La -g_1h_{j_1}$,
where $g_1= \frac 1{x+c_1}$, see formula \Ref{T}.
We have
\bean
\label{m=1r}
\frac{\der}{\der t_r}\Big|_{\mu^J(c_1)} =-\frac{\der}{\der x}
\Big[e^{g_1f_{j_1}} \La^r e^{-g_1 f_{j_1}}\Big]^0.
\eean
Assume  $j_1\in\{1,\dots,n-1\}$. Then
$e^{g_1f_{j_1}} = 1+g_1(e_{j_1,j_1}+e_{2n+1-j_1,2n+1-j_1})\Lambda^{-1}$.
If $r$ is odd and $r>1$, then the right-hand side of \Ref{m=1r} is zero, hence
 $\frac{\der}{\der t_r}|_{\mu^J(c_1)}=\Ga_r=0$. Let $r=1$, then
\bea
\frac{\der}{\der t_1}\Big|_{\mu^J(c_1)} =
-\frac{\der}{\der x}
\Big[e^{g_1f_{j_1}} \La e^{-g_1 f_{j_1}}\Big]^0 =\frac{\der}{\der x}g_1h_{j_1}= -\frac 1{(x+c_1)^2}h_{j_1} = -\frac {\der\mu^J}{\der c_1} (c_1).  
\eea
Hence $\Ga_1=-\frac{\der}{\der c_1}$.
Assume $j_1=n$, then
	by formula \Ref{T mkdv},
	\bean
\label{La r>1}
		\frac \der{\der t_r}\Big|_{\mu^J(c_1)}
	&=&
 - \frac {\der}{\der x}\Big[(1 + 2g_1(e_{n,n}+e_{n+1,n+1})\La^{-1} + 4g_1^2 e_{n,n} \La^{-2})
	\\
\notag
	&\times&
 \Lambda^r (1 - 2g_1(e_{n,n}+e_{n+1,n+1})\La^{-1} + 4g_1^2 e_{n,n} \La^{-2}\Big]^0.
	\eean
	It follows from \Ref{La r>1} and 
Lemma \ref{lem lambda} that $\frac \der{\der t_r}\big|_{\mu^J(c_1)}=0$ for $r$ odd and $r>1$, hence
	$\Ga_r=0$. For $r=1$ we have 
$\frac \der{\der t_r}\big|_{\mu^J(c_1)}= -2\frac{ dg_1}{dx} (e_{n,n}-e_{n+2,n+2}) = \frac {dg_1}{dx} h_n=
-\frac 1{(x+c_1)^2}h_n=- \frac {\der\mu^{J}}{\der c_1} (c_1)$. Hence $\Ga_1=-\frac\der{\der c_1}$.
Similarly, if $j_1=0$, then $\Ga_1=-\frac\der{\der c_1}$ and $\Ga_r=0$ for $r>1$.
	Theorem \ref{thm main} is proved for $m=1$.
	
	\subsection{Beginning of proof of Theorem \ref{thm main} for $m>1$}
		
	We prove the first statement of Theorem \ref{thm main} by induction on $m$. Let $J=(j_1,\dots,j_m)$.
	Assume that the statement is proved for ${\tilde J}=(j_1,\dots,j_{m-1})$.
	Let
	\bea
	Y^{{\tilde J}}\ : \ \C^{m-1} \to  (\C[x])^{n+1}, \quad
	{\tilde c}=(c_1,\dots,c_{m-1}) \ \mapsto \ (y_0(x,{\tilde c}),\dots, y_n(x,{\tilde c}))
	\eea
	be the generation of tuples in the ${\tilde J}$-th direction. Then the generation
	of tuples in the $J$-th direction is
	\bea
	Y^{J}\ :\ \C^m \mapsto (\C[x])^{n+1}, \quad
	({\tilde c},c_m) \mapsto
	\ (y_0(x,\tilde c),\dots,  y_{j_m,0}(x,{\tilde c}) + c_m y_{j_m}(x,{\tilde c}),\dots, y_n(x,\tilde c),
	\eea
	see \Ref{J'} and \Ref{Ja}.
	We have $g_m = \ln'(y_{j_m,0}(x,{\tilde c}) + c_m y_{j_m}(x,{\tilde c}))- \ln' (y_{j_m}(x,{\tilde c}))$,
	see \Ref{g's}.

	By the induction assumption,
	there exists a polynomial vector field $\Gamma_{r,{\tilde J}}=\sum_{i=1}^{m-1}\ga_i(\tilde c)\frac\der{\der c_i}$ on $\C^{m-1}$ such that 	for all ${\tilde c}\in\C^{m-1}$ we have
	\bean
	\label{fromula main m-1}
	\frac \der{\der t_r}\Big|_{\mu^{{\tilde J}}({\tilde c})} = 
\frak{L}_{\Ga_{r,{\tilde J}}}\mu^{{\tilde J}}(\tilde c).
	\eean

	\begin{prop}
		\label{prop ind}
		There exists a scalar polynomial
		$\ga_{m}(\tilde c,c_m)$ on $\C^m$  such that the vector field
		$\Ga_r = \Gamma_{r,{\tilde J}} + \ga_{m}({\tilde c},c_m)\frac \der{\der c_m}$
		satisfies \Ref{formula main} for all $(\tilde c,c_m)\in\C^m$.
	\end{prop}

	\subsection{Proof of Proposition \ref{prop ind}}
	\label{sec proof prop ind}

\begin{lem}
\label{lem ker 1}
We have 
	\bean
		\label{formula for ker}
		d\frak m_i\Big|_{\mu^J({\tilde c},c_m)} \left(\frac{\der}{\der t_r}\Big|_{\mu^J({\tilde c},c_m)} -
\frak{L}_{ \Ga_{r,\tilde J}} \mu^J(\tilde c,c_m)\right) = 0.
		\eean
 for all $i\notin \{j_m, 2n+1-j_m\}$.

\end{lem}
		
\begin{proof}
The proof  is the same as the proof of Lemma 5.5 in \cite{VW}.
Namely, by Theorem \ref{thm gaugemiura} we have $\frak m_i \circ \mu^J(\tilde c,c_m)
=
\frak m_i \circ \mu^{\tilde J}(\tilde c)$ for all 
$i\notin \{j_m, 2n+1-j_m\}$.
Hence,
\bean
\label{Ga inv}
d\frak m_i\big|_{\mu^J({\tilde c},c_m)} \left(
\frak{L}_{ \Ga_{r,\tilde J}} \mu^J(\tilde c,c_M)\right)
=
\frak{L}_{ \Ga_{r,\tilde J}} (\frak m_i \circ\mu^J) ({\tilde c},c_m)=
\frak{L}_{ \Ga_{r,\tilde J}} (\frak m_i \circ\mu^{\tilde J}) ({\tilde c}).
 \eean
By Theorems \ref{thm mkdvtokdv} and \ref{thm gaugemiura}, we have
\bean
\label{d/dt inv}
d\frak m_i\Big|_{\mu^J({\tilde c},c_m)} \left(\frac{\der}{\der t_r}\Big|_{\mu^J({\tilde c},c_m)} \right) =
\frac{\der}{\der t_r}\Big|_{\frak m_i \circ\mu^J({\tilde c},c_m)} =
 \frac{\der}{\der t_r}\Big|_{\frak m_i \circ\mu^{{\tilde J}}({\tilde c})}.
\eean
By the induction assumption, we have
\bean
\label{formula Ga t}
\frac{\der}{\der t_r}\Big|_{\frak m_i \circ\mu^{{\tilde J}}({\tilde c})} =
\frak{L}_{ \Ga_{r,\tilde J}} (\frak m_i \circ\mu^{\tilde J}) ({\tilde c}).
\eean
These three formulas prove the lemma.
\end{proof}
	
\begin{lem}
\label{lem differ}
The difference
$\frac {\der}{\der t_r}\big|_{\mu^J(c)} - 
\frak{L}_{ \Ga_{r,\tilde J}} \mu^J ({\tilde c},c_m)$
has the form indicated in the right-hand side of formula \Ref{k1} if 
$j_m\in\{1,\dots,n-1\}$,
has the form indicated in the right-hand side of formula \Ref{k2} if 
$j_m=0$,
has the form indicated in the right-hand side of formula \Ref{k3} if 
$j_m=n$.

\end{lem}

\begin{proof} We have
\bea
&&
\frac {\der}{\der t_r}\Big|_{\mu^J(c)} - \frak{L}_{ \Ga_{r,\tilde J}} \mu^J ({\tilde c},c_m)
\\
&&
\phantom{aaaaaaaaaa}
=
\frac {\der}{\der t_r}\Big|_{\mu^J(c)} - \frac {\der}{\der t_r}\Big|_{\mu^{\tilde{J}}(\tilde c)} 
 + \frac {\der}{\der t_r}\Big|_{\mu^{\tilde{J}}(\tilde c)}
 - \frak{L}_{ \Ga_{r,\tilde J}} \mu^J ({\tilde c},c_m)
\\
&&
\phantom{aaaaaaaaaa}
=
\frac {\der}{\der t_r}\Big|_{\mu^J(c)} - \frac {\der}{\der t_r}\Big|_{\mu^{\tilde{J}}(\tilde c)} 
 +\frak{L}_{ \Ga_{r,\tilde J}} \mu^{\tilde J} ({\tilde c})
 - \frak{L}_{ \Ga_{r,\tilde J}} \mu^J ({\tilde c},c_m)
\\
&&
\phantom{aaaaaaaaaa}
=
\frac {\der}{\der t_r}\Big|_{\mu^J(c)} - \frac {\der}{\der t_r}\Big|_{\mu^{\tilde{J}}(\tilde c)} 
 +\frak{L}_{ \Ga_{r,\tilde J}}g_m(x,{\tilde c},c_m)\,h_{j_m},
\eea
see formula \Ref{oper2}.  If $j_m\in\{1,\dots,n-1\}$, then 
$\frac {\der}{\der t_r}\Big|_{\mu^J(c)} - \frac {\der}{\der t_r}\Big|_{\mu^{\tilde{J}}(\tilde c)} $
has the form indicated in the right-hand side of formula \Ref{k1} by Proposition \ref{prop77}
and $\frak{L}_{ \Ga_{r,\tilde J}}g_m(x,{\tilde c},c_m)\,h_{j_m}$
has that form since 
$h_{j_m} = -e_{j_m, j_m} + e_{j_m+1, j_m+1}-e_{2n+1-j_m, 2n+1-j_m}
+e_{2n+2-j_m, 2n+2-j_m}$. 
 This proves the lemma for $j_m\in\{1,\dots,n-1\}$. The other two cases of the lemma are proved similarly.
\end{proof}

Let us finish the proof of Proposition \ref{prop ind}.
By Lemmas \ref{lem ker 1} and \ref{lem differ}
the difference $\frac {\der}{\der t_r}\big|_{\mu^J(c)} -
\frak{L}_{ \Ga_{r,\tilde J}} \mu^J ({\tilde c},c_m)$
has the form indicated in the right-hand side of one of the formulas \Ref{k1}-\Ref{k3} and lies in the kernels of the differentials
of Miura maps
$\frak{ m}_i$ for all $i\notin \{j_m, 2n+1-j_m\}$.  By Propositions \ref{prop 6.8}, \ref{prop 6.13}, \ref{prop n}
 we conclude that the difference has the form 
$\ga_m(\tilde c, c_m) \frac{\der \mu^J}{\der c_m}$ for some scalar	function $\ga_m(\tilde c, c_m) $
 on  $\C^m$. Therefore, $\frac{\der}{\der t_r}\big|_{\mu^J(\tilde c, c_m)} = 
\frak{L}_{ \Ga_{r,\tilde J}} \mu^J ({\tilde c},c_m)+ 
\\
\ga_m(\tilde c, c_m) \frac{\der \mu^J}{\der c_m}(\tilde c, c_m)$.
If we set $\Ga_r = \Ga_{r,\tilde J} +
\ga_m(\tilde c, c_m) \frac{\der}{\der c_m}$, then the vector field  $\Ga_r$ will satisfy formula 
\Ref{formula main}.

We need to prove that $\ga_m(\tilde c, c_m) $ is a polynomial. The proof of that fact is the same as the proof of 
\cite[Proposition 5.9]{VW}.
Proposition \ref{prop ind} is proved.

\subsection{End of proof Theorem \ref{thm main} for $m>1$}
Proposition \ref{prop ind} implies  the first statement of Theorem \ref{thm main}. The second statement 
 says that if $r>4m$, then
$		\frac \der{\der t_r}\big|_{\mu^J(c)} =0$. But that follows from Corollary \ref{cor der} and Lemma \ref{lem exp}.

\end{document}